\documentclass[12pt,a4paper]{article}
\usepackage[english]{babel}
\usepackage[utf8]{inputenc}
\usepackage[T1]{fontenc}
\usepackage{lmodern}
\usepackage{amsmath, amssymb, amsthm}
\usepackage{mathrsfs}
\usepackage{amsfonts}
\usepackage[top=2.5cm, bottom=2.5cm, left=2.5cm, right=2.5cm]{geometry}
\usepackage{hyperref}
\usepackage{enumitem}
\usepackage{bm}
\usepackage{indentfirst}

\numberwithin{equation}{section}

\theoremstyle{plain}
\newtheorem{theorem}{Theorem}[section]
\newtheorem{lemma}{Lemma}[section]
\newtheorem{proposition}{Proposition}[section]

\theoremstyle{definition}
\newtheorem{definition}{Definition}[section]
\newtheorem{example}{Example}[section]
\newtheorem{remark}{Remark}[section]

\newcommand{\R}{\mathbb{R}}

\newcommand{\abs}[1]{\left|#1\right|}
\newcommand{\norm}[1]{\left\|#1\right\|}
\newcommand{\eps}{\varepsilon}
\newcommand{\dd}{\,\mathrm{d}}
\newcommand{\J}{\mathcal{J}}
\newcommand{\inner}[2]{\left\langle #1,#2\right\rangle}

\title{Brezis--Nirenberg-type results for coupled
critical elliptic systems with sign-changing perturbations
\thanks{Research supported by Natural Science Foundation of Shandong Province (No. ZR2024MA009).}}
\author{Guanwei Chen\\
\small  \textit{School of Mathematical Sciences, University of Jinan,}\\ \small \textit{Jinan 250022, Shandong Province, China}\\
\small E-mail: guanweic@163.com}
\date{}

\begin{document}
\maketitle

\begin{abstract}
We establish Brezis--Nirenberg-type existence results for elliptic
systems with symmetric linear coupling, a radial critical
nonlinearity and a subcritical potential.
For any dimension $N\ge3$, the potential may change sign, with its negative
part bounded by a subquadratic power for some exponent
$q_1\in(1,2)$ and no smallness restriction on the coefficient.
Under coercivity and suitable near-zero and subcritical growth
assumptions, we obtain a nontrivial weak solution if the largest
eigenvalue of the coupling matrix is positive when $N\ge4$,
or exceeds a positive constant depending on the domain when $N=3$.
No homogeneity, convexity or monotonicity along rays is required
of the perturbation potential, and either sign of the linear
coupling is allowed.
The proof combines a variational criterion below the critical
Sobolev energy level with estimates showing that the negative
part of the potential contributes less than the linear energy
decrease.
On star-shaped domains, we also obtain a Pohozaev nonexistence
criterion.
\end{abstract}



\section{Introduction}\label{sec:intro}

For subcritical elliptic systems, a substantial body of results
has been established on the existence and multiplicity of solutions.
For subquadratic or asymptotically quadratic potentials,
existence and multiplicity results, including resonant cases,
were obtained in \cite{Costa1,Ma-6,Ma-7,Pomponio}.
Infinitely many solutions
for systems with subquadratic or superquadratic potentials were obtained in \cite{Chen-2}.
Ground state solutions in the whole space were obtained in
\cite{Chen} for asymptotically quadratic or superquadratic potentials.
More recently, Morse theory was used in \cite{RecovaRumbos2026} to obtain multiple solutions for superquadratic
potentials.

At the critical exponent $2^*=2N/(N-2)$, the embedding
$H_0^1(\Omega)\hookrightarrow L^{2^*}(\Omega)$ is not compact.
For the scalar Brezis--Nirenberg problem
\[
-\Delta u=\lambda u+|u|^{2^*-2}u
\quad\text{in }\Omega,\qquad
u=0\quad\text{on }\partial\Omega,
\]
positive solutions exist for $0<\lambda<\lambda_1$ when
$N\ge4$, and exactly for $\lambda_1/4<\lambda<\lambda_1$
on a three-dimensional ball, as proved in \cite{BrezisNirenberg}.
For variable linear potentials on three-dimensional domains,
a Green function criterion for the strict Sobolev quotient
inequality was established in \cite{Druet2002}, and perturbation
energy expansions were obtained in \cite{FrankKonigKovarik2021}.
On the three-dimensional unit ball, sign-changing solutions
were constructed in \cite{SunWeiYang2026} for sufficiently small
$\lambda>0$, and infinitely many nonradial ones in
\cite{SunWeiYang2025} for every $\lambda>0$.

For critical elliptic systems, existence, multiplicity,
least energy solutions and phase separation have been studied
for several coupling structures.
Critical Schr\"odinger systems with a diagonal linear part and
homogeneous nonlinear coupling are treated in
\cite{ChenZou2012,ChenZou2015,ClappSzulkin2019,LiuZou2025,
TavaresYou2020,TavaresYouZou2022,YinZou2022},
with results on least energy solutions and phase separation under
cooperative, competitive or mixed nonlinear interactions.
Related results with Neumann boundary conditions appear in
\cite{MauroSchieraTavares2026}, where the critical case is treated
in dimension four.
For systems with linear coupling and componentwise critical powers,
existence, ground state and multiplicity results are obtained in
\cite{PengShuaiWang2017,Wu2019}.
Positive ground states and their concentration in steep potential
wells are studied in \cite{LinWangWu2024}, while positive solution
multiplicity and asymptotic behavior under variable linear coupling
are considered in \cite{WuWang2025}.
In these models, the nonlinearities act separately on the components.
Both linear and nonlinear coupling occur in
\cite{AlvesMoraisSouto2000}, whose critical potential is proportional
to $|u|^\alpha|v|^\beta$, with $\alpha,\beta>1$ and
$\alpha+\beta=2^*$, and vanishes whenever either component is zero.
For unperturbed systems with critical potentials given by powers
of norms, including those with the radial critical term
$|U|^{2^*-2}U$, existence results are obtained in
\cite{AmsterDeNapoliMariani2002} under the assumption that the
absolute value of every eigenvalue of the symmetric coupling
matrix is smaller than the first Dirichlet eigenvalue.
The homogeneous frameworks of
\cite{BarbosaMontenegro2011,Gluck2025} cover the unperturbed
radial system for $N\ge4$ when the largest eigenvalue of the
coupling matrix is positive and smaller than the first Dirichlet
eigenvalue, with homogeneity required of the lower-order potentials.
For critical $p$-Laplacian systems, the multiplicity results in
\cite{ChuLeiSunSuo2017} assume homogeneous positive and negative
lower-order potentials with small coefficients.
The linking results in \cite{GlossRibeiroMarinhoPerera2026}
extend beyond the coercive range but require a positive, even,
convex $p$-homogeneous potential and a nonnegative subcritical
potential satisfying an Ambrosetti--Rabinowitz-type condition.

{\bf The above results motivate the question} of whether the sufficient
spectral conditions for the unperturbed radial critical system
persist under general sign-changing subcritical perturbations.
We allow the perturbation potential to depend on $x$ and both
components, without homogeneity, convexity, evenness or
monotonicity along rays.
The symmetric linear coupling matrix is subject to coercivity
of the quadratic part, with no lower bound on its negative
eigenvalues and either sign of the off-diagonal coefficient.
In dimensions $N\ge3$, the potential may change sign,
provided that it satisfies the mild near-zero and subcritical
growth assumptions below and its negative part admits a
subquadratic power bound for some $q_1\in(1,2)$,
with no smallness restriction on the coefficient.

More precisely, we consider the critical elliptic system
\begin{equation}\label{sys1}
\left\{
\begin{array}{ll}
-\Delta u=\lambda u+\delta v+u(u^2+v^2)^{\frac{2^*-2}{2}}+f(x,u,v), & x\in\Omega,\\[2mm]
-\Delta v=\delta u+\gamma v+v(u^2+v^2)^{\frac{2^*-2}{2}}+g(x,u,v), & x\in\Omega,\\[2mm]
u=v=0, & x\in\partial\Omega,
\end{array}
\right.
\end{equation}
where $\Omega\subset\R^N$ is a bounded smooth domain, $N\ge3$,
$2^*:=\frac{2N}{N-2}$, $\lambda,\delta,\gamma\in\R$, and $(f,g)=\nabla F$
for some $F$, where $\nabla F$ denotes the gradient with respect to $U=(u,v)$.
The coupling matrix of the system \eqref{sys1} is
$A^*=\begin{pmatrix}\lambda & \delta\\ \delta & \gamma\end{pmatrix}$,
with eigenvalues
$\xi,\zeta=\frac{\lambda+\gamma}{2}
\pm\sqrt{\left(\frac{\lambda-\gamma}{2}\right)^2+\delta^2}.$

For the largest eigenvalue $\xi$ of $A^*$, we assume that
\begin{equation}\label{posdef}
\xi<\lambda_1,
\end{equation}
where $\lambda_1>0$ is the first Dirichlet eigenvalue of $-\Delta$ on $\Omega$.

We assume that $F\in C^1(\overline\Omega\times\R^2,\R)$
satisfies the following conditions:
\begin{enumerate}[label=,leftmargin=*]
\item (F1) $F(x,0)=0$ and $\nabla F(x,0)=0$ for all $x\in\Omega$;
\item (F2) $\displaystyle \lim_{|U|\to0}\frac{|\nabla F(x,U)|}{|U|}=0$ uniformly in $x\in\Omega$;
\item (F3) $\displaystyle \lim_{|U|\to+\infty}\frac{|\nabla F(x,U)|}{|U|^{2^*-1}}=0$ uniformly in $x\in\Omega$;
\item (F4) There exist $C_1>0$, $q_1\in(1,2)$ such that
$F(x,U)\ge -C_1|U|^{q_1}$, $\forall(x,U)\in\Omega\times\R^2$.
\end{enumerate}
Condition (F4) restricts only the negative part of $F$;
the positive part may grow superquadratically subject to (F3).

\begin{example}
\label{rem:F4-compatible}
Assumptions (F1)--(F4) admit a broad class of potentials,
including $F\equiv0$ and nonnegative or rotationally invariant
examples. To illustrate their scope in every dimension $N\ge3$,
we consider the following sign-changing, nonconvex potential
without rotational invariance:
\[
F(x,U)=-a(x)|U|^{q_1}\chi(|U|)
+b(x)(u^2+2v^2)^{q_2/2},
\qquad U=(u,v),
\]
where $a,b\in C^1(\overline\Omega)$, $a,b\ge0$,
$1<q_1<2<q_2<2^*$,
$\chi\in C^\infty([0,\infty))$, $\chi\equiv0$ on $[0,1]$,
$\chi\equiv1$ on $[2,\infty)$, and $0\le\chi\le1$.
It satisfies (F1)--(F4).
If $a(x)>2^{q_2-q_1}b(x)>0$
for all $x\in\overline\Omega$, then $F(x,U)>0$
for $0<|U|\le1$, whereas
$F(x,2,0)=-2^{q_1}a(x)+2^{q_2}b(x)<0$.
Moreover, since $u^2+2v^2\ge|U|^2$ and
$\min_{\overline\Omega}b>0$, we have $F(x,U)>0$
for sufficiently large $|U|$, uniformly in $x$.
Thus $F$ changes sign. Besides, $F(x,\cdot)$ is neither convex nor
rotationally invariant.
\end{example}

\begin{definition}
\label{def:W}
Let $W:=H_0^1(\Omega)\times H_0^1(\Omega)$. For $U=(u_1,u_2)$ and
$V=(v_1,v_2)$ in $W$, write $|U|:=(u_1^2+u_2^2)^{1/2}$ and
$|\nabla U|^2:=|\nabla u_1|^2+|\nabla u_2|^2$, and endow $W$ with the
inner product and the corresponding norm
\[
\inner{U}{V}_W:=\int_\Omega\nabla U\cdot\nabla V\dd x,\quad \|U\|_W^2:=\int_\Omega|\nabla U|^2\dd x.
\]
For $u\in H_0^1(\Omega)$ we
use the equivalent norm $\|u\|_{H_0^1}^2:=\int_\Omega|\nabla u|^2\dd x$.
The best constant in the scalar Sobolev embedding
$H_0^1(\Omega)\hookrightarrow L^{2^*}(\Omega)$ is
\[
S:=\inf_{u\in H_0^1(\Omega)\setminus\{0\}}
\frac{\|\nabla u\|_{L^2}^2}{\|u\|_{L^{2^*}}^2}.
\]
\end{definition}

The energy functional associated with \eqref{sys1} is
$\J:W\to\R$, defined by
\begin{equation}\label{J-intro}
\J(U)=\frac12\norm{U}_W^2
-\frac12\int_\Omega A^*U\cdot U\dd x
-\frac{1}{2^*}\int_\Omega|U|^{2^*}\dd x
-\int_\Omega F(x,U)\dd x,
\qquad U\in W.
\end{equation}
Under (F1)--(F3), $\J\in C^1(W,\R)$ with
\begin{equation}\label{deriv}
\inner{\J'(U)}{V}
=\inner{U}{V}_W
-\int_\Omega A^*U\cdot V\dd x
-\int_\Omega\abs{U}^{2^*-2}U\cdot V\dd x
-\int_\Omega\nabla F(x,U)\cdot V\dd x,
\end{equation}
for all $U,V\in W$. Critical points of $\J$ are weak solutions of \eqref{sys1}.
A solution is called nontrivial if $U\not\equiv0$.

We first establish an abstract existence theorem for nontrivial
weak solutions of \eqref{sys1} under the energy threshold
condition \eqref{energy-cond} (Theorem~\ref{thm:abstract}).
We then verify this condition for $N\ge4$ and $N=3$, obtaining
the existence results in Theorems~\ref{thm:N4} and
\ref{thm:N3}, respectively.
On star-shaped domains, we also establish a Pohozaev
nonexistence criterion under an additional differential bound
on the potential (Proposition~\ref{prop:nonexistence}).

\begin{theorem}\label{thm:abstract}
Let $N\ge 3$. Assume \eqref{posdef} and (F1)--(F3), and suppose
\begin{equation}\label{energy-cond}
\exists \ U_0\in W\setminus\{0\} \ \ \mbox{s.t.} \quad \sup_{t\ge0}\J(tU_0)<\frac1N S^{N/2}.
\end{equation}
Then \eqref{sys1} admits at least one nontrivial weak solution
$U\in W$.
\end{theorem}

\begin{theorem}\label{thm:N4}
Let $N\ge4$. Assume \eqref{posdef},
$\xi>0$ and (F1)--(F4). Then \eqref{energy-cond}
holds, and consequently \eqref{sys1} admits at least one nontrivial
weak solution $U\in W$.
\end{theorem}

For $N=3$, the required test-function estimate does not follow
from $\xi>0$ alone. Let $\lambda_{\mathrm{BN}}$ denote the
threshold $\lambda^*\in(0,\lambda_1)$ specified in
\cite[Corollary 1.1 and (1.51)]{BrezisNirenberg}.
It depends on $\Omega$ and equals $\lambda_1/4$ when
$\Omega$ is a ball.

\begin{theorem}\label{thm:N3}
Let $N=3$. Assume \eqref{posdef},
$\xi>\lambda_{\mathrm{BN}}$ and (F1)--(F4), where
$\lambda_{\mathrm{BN}}$ is as defined above. Then \eqref{energy-cond} holds, and consequently \eqref{sys1} admits
at least one nontrivial weak solution $U\in W$.
\end{theorem}

We also obtain a supplementary nonexistence result. For $x_0\in\Omega$,
set
\begin{equation}\label{eq:pohozaev-density}
\mathcal P_F(x,U):=NF(x,U)+(x-x_0)\cdot\nabla_xF(x,U)
-\frac{N-2}{2}\nabla F(x,U)\cdot U,
\end{equation}
where $\nabla_xF$ denotes differentiation with $U$ held fixed.

\begin{proposition}\label{prop:nonexistence}
Let $N\ge3$. Assume (F1)--(F3) and that $\Omega$ is star-shaped
with respect to $x_0\in\Omega$. Suppose there exists a constant
$\kappa_F\ge0$ such that
\begin{equation}\label{eq:pohozaev-bound}
\mathcal P_F(x,U)\le\kappa_F|U|^2
\qquad\text{for all }(x,U)\in\Omega\times\R^2.
\end{equation}
If $\xi+\kappa_F<0$, then \eqref{sys1} has only the trivial weak
solution in $W$.
\end{proposition}

The following table summarizes our existence and nonexistence
results.

\begin{center}
\small
\renewcommand{\arraystretch}{1.3}
\setlength{\tabcolsep}{4pt}
\begin{tabular}{|c|c|c|c|c|}
\hline
$N$ & Assumptions on $F$ & Bounded smooth domain $\Omega$
& Condition on $\xi$ & Reference \\
\hline
\multicolumn{5}{|c|}{Existence of a nontrivial weak solution} \\
\hline
$\ge4$ & (F1)--(F4) & any
& $0<\xi<\lambda_1$ & Thm.~\ref{thm:N4} \\
\hline
$=3$ & (F1)--(F4) & a ball
& $\lambda_1/4<\xi<\lambda_1$ & Thm.~\ref{thm:N3} \\
\hline
$=3$ & (F1)--(F4) & not a ball
& $\lambda^*<\xi<\lambda_1$ & Thm.~\ref{thm:N3} \\
\hline
\multicolumn{5}{|c|}{Nonexistence of nontrivial weak solutions} \\
\hline
$\ge3$
& \begin{tabular}{@{}c@{}}
(F1)--(F3),
\eqref{eq:pohozaev-bound}
\end{tabular}
& star-shaped
& $\xi+\kappa_F<0$
& Prop.~\ref{prop:nonexistence} \\
\hline
\end{tabular}
\end{center}

\begin{remark} The main difficulties are as follows. \textbf{(1)}
Coupling in \eqref{sys1} may occur through the linear term
$A^*U$, the critical nonlinearity $|U|^{2^*-2}U$, and the
perturbation $\nabla F$. The matrix $A^*$ need not be positive
definite, and the general perturbation prevents a reduction
of the system to independent scalar equations.
\textbf{(2)}
The potential $F$ need not be homogeneous, even, or monotone
along rays, and may change sign in every dimension under
consideration. Its negative part raises the test-function
energy. A central difficulty is therefore to verify the
energy threshold condition \eqref{energy-cond} for our system \eqref{sys1}.
\textbf{(3)}
For $N=3$, the cutoff error and the linear contribution to
the Sobolev quotient are both of order $\eps^{1/2}$.
Their coefficients determine whether the quotient falls below
$S$, making the estimate sensitive to the linear parameter
and the geometry of the domain.
\end{remark}

\begin{remark} The main contributions of the paper are as follows.\\
\textbf{(1) Model.}
The potential $F(x,U)$ may be nonseparable and sign-changing
for every $N\ge3$, without homogeneity, convexity, evenness
or monotonicity along rays. Its negative part satisfies (F4)
with no smallness restriction on $C_1$.
The linear assumptions allow either sign of $\delta$ and impose
no lower bound on the negative eigenvalues of $A^*$.\\
\textbf{(2) Method.}
Theorem \ref{thm:abstract} extends the scalar criterion in
\cite[Theorem 2.1]{BrezisNirenberg} to the present system.
Its proof obtains a nonzero weak limit of a mountain pass
sequence below $S^{N/2}/N$.
The energy condition is verified using test functions along
an eigenvector corresponding to $\xi$, without rotational
invariance of $F$.
The key estimates show that the negative part of $F$ contributes
at a lower order than the linear energy decrease, uniformly
on bounded portions of the test rays.\\
\textbf{(3) Results.}
Under (F1)--(F4), the sufficient spectral ranges are
$0<\xi<\lambda_1$ for $N\ge4$ and
$\lambda_{\mathrm{BN}}<\xi<\lambda_1$ for $N=3$.
These ranges remain sufficient for sign-changing potentials
in every dimension $N\ge3$, with no smallness restriction
on the coefficient controlling their negative parts.
The results also allow both diagonal coefficients to be
nonpositive.
Proposition~\ref{prop:nonexistence} gives a supplementary
Pohozaev nonexistence criterion on star-shaped domains under
an additional differential bound on $F$.
\end{remark}

The paper is organized as follows. Section \ref{sec:framework}
introduces the variational framework and gives some preliminary lemmas. The abstract existence result (Theorem \ref{thm:abstract}) is proved
in Section \ref{sec:abstract}. Sections \ref{sec:N4} and
\ref{sec:N3} verify \eqref{energy-cond} for $N\ge4$ and $N=3$,
respectively, and thereby establish Theorems
\ref{thm:N4} and \ref{thm:N3}. Section \ref{sec:nonexistence}
proves Proposition \ref{prop:nonexistence}, including the regularity
needed for the Pohozaev identity.

\section{Variational framework and preliminary lemmas}\label{sec:framework}

The following result follows from the Sobolev and
Rellich--Kondrachov embedding theorems; see \cite{BrezisFA}.

\begin{lemma}\label{rem:W-embeddings}
The embedding
$W\hookrightarrow L^s(\Omega)\times L^s(\Omega)$
is continuous for $1\le s\le2^*$ and compact for
$1\le s<2^*$. It is not compact at $s=2^*$.
\end{lemma}

\begin{lemma}\label{lem:positive}
Under \eqref{posdef} there exists $c_0>0$ such that
\begin{equation}\label{coercive}
\norm{U}_W^2-\int_\Omega A^*U\cdot U\dd x\ge c_0\norm{U}_W^2,\quad U\in W.
\end{equation}
\end{lemma}

\begin{proof}
Let $\vec e_1,\vec e_2$ be orthonormal eigenvectors of $A^*$
corresponding to $\xi,\zeta$, respectively. Since $A^*$ is symmetric,
such a basis exists. For $U\in W$, write
\[
U=\alpha_1\vec e_1+\alpha_2\vec e_2,\qquad
\alpha_i:=U\cdot\vec e_i.
\]
As $\vec e_i$ are constant vectors, $\alpha_i\in H_0^1(\Omega)$ for
$i=1,2$. Moreover, since $\vec e_i\cdot\vec e_j=\delta_{ij}$ ($\delta_{ij}=0$ if $i\neq j$, and $\delta_{ij}=1$ if $i= j$), we have
\[
\nabla U=\vec e_1\otimes\nabla\alpha_1+\vec e_2\otimes\nabla\alpha_2,
\]
and hence
\[
|\nabla U|^2
=|\vec e_1|^2|\nabla\alpha_1|^2
+2(\vec e_1\cdot\vec e_2)(\nabla\alpha_1\cdot\nabla\alpha_2)
+|\vec e_2|^2|\nabla\alpha_2|^2
=|\nabla\alpha_1|^2+|\nabla\alpha_2|^2.
\]
Integrating over $\Omega$,
\[
\norm{U}_W^2
=\int_\Omega|\nabla U|^2\dd x
=\|\alpha_1\|_{H_0^1}^2+\|\alpha_2\|_{H_0^1}^2.
\]
Similarly, using $A^*\vec e_i=\xi_i\vec e_i$ with $\xi_1:=\xi$ and
$\xi_2:=\zeta$, together with $\vec e_i\cdot\vec e_j=\delta_{ij}$,
\[
\int_\Omega A^*U\cdot U\dd x
=\xi\int_\Omega\alpha_1^2\dd x+\zeta\int_\Omega\alpha_2^2\dd x.
\]
Hence
\[
\norm{U}_W^2-\int_\Omega A^*U\cdot U\dd x
=\sum_{i=1}^{2}\left(\|\alpha_i\|_{H_0^1}^2-\xi_i\|\alpha_i\|_{L^2}^2\right).
\]
If $\xi_i\ge0$, the Poincar\'e inequality gives
\[
\|\alpha_i\|_{H_0^1}^2-\xi_i\|\alpha_i\|_{L^2}^2
\ge\left(1-\frac{\xi_i}{\lambda_1}\right)
\|\alpha_i\|_{H_0^1}^2.
\]
If $\xi_i<0$, instead
\[
\|\alpha_i\|_{H_0^1}^2-\xi_i\|\alpha_i\|_{L^2}^2
\ge\|\alpha_i\|_{H_0^1}^2.
\]
Consequently, in both cases,
\[
\|\alpha_i\|_{H_0^1}^2-\xi_i\|\alpha_i\|_{L^2}^2
\ge c_0\|\alpha_i\|_{H_0^1}^2,
\qquad
c_0:=\min\left\{1,1-\frac{\xi}{\lambda_1},
1-\frac{\zeta}{\lambda_1}\right\}>0,
\]
where positivity follows from \eqref{posdef}. Summing over $i$,
\[
\norm{U}_W^2-\int_\Omega A^*U\cdot U\dd x\ge c_0\sum_{i=1}^2\|\alpha_i\|_{H_0^1}^2=c_0\norm{U}_W^2,
\]
which is \eqref{coercive}.
\end{proof}

\begin{lemma}\label{lem:growth}
Assume (F1)--(F3). Then for every $\eps>0$ there exists $M=M(\eps)>0$ such that for a.e.\ $x\in\Omega$ and all $U\in\R^2$,
\begin{equation}\label{est1}
|\nabla F(x,U)|\le 2^*\eps |U|^{2^*-1}+M(\eps),
\end{equation}
and
\begin{equation}\label{est2}
|F(x,U)|\le \eps|U|^{2^*}+M(\eps)|U|.
\end{equation}
\end{lemma}

\begin{proof}
By (F3), for every $\eps>0$ there exists $\sigma=\sigma(\eps)>0$ such that
$$
|\nabla F(x,U)|\le 2^*\eps|U|^{2^*-1},
$$
for $|U|>\sigma$ and a.e.\ $x\in\Omega$. Since $\nabla F$ is continuous on the compact set
$\overline\Omega\times\{U\in\R^2:|U|\le\sigma\}$,
there exists $M=M(\eps)>0$ such that
$$
|\nabla F(x,U)|\le M,
$$
for $|U|\le\sigma$ and a.e.\ $x\in\Omega$. Combining the two estimates gives \eqref{est1}. Integrating \eqref{est1} from $0$ to $U$ (using $F(x,0)=0$) gives \eqref{est2}.
\end{proof}

\begin{lemma}\label{lem:growth2}
Assume (F1)--(F3). Then for every $\eps>0$ there exists $d(\eps)>0$ such that for a.e.\ $x\in\Omega$ and all $U\in\R^2$,
\begin{equation}\label{est3}
|\nabla F(x,U)|\le 2\eps|U|+2^*d(\eps)|U|^{2^*-1},
\end{equation}
and
\begin{equation}\label{est4}
|F(x,U)|\le \eps|U|^2+d(\eps)|U|^{2^*}.
\end{equation}
\end{lemma}

\begin{proof}
By (F2), for every $\eps>0$ there exists $\sigma_1=\sigma_1(\eps)>0$ such that
\begin{equation}\label{sigma}
|\nabla F(x,U)|\le 2\eps|U|
\qquad\text{for }|U|<\sigma_1\text{ and a.e. }x\in\Omega.
\end{equation}
By \eqref{est1} with $\eps=1$ there exists $M(1)>0$ such that
\begin{equation}\label{est1-eps1}
|\nabla F(x,U)|\le 2^*|U|^{2^*-1}+M(1)
\end{equation}
for a.e.\ $x\in\Omega$ and all $U\in\R^2$.
For $|U|\ge\sigma_1$, we have $|U|^{2^*-1}\ge\sigma_1^{2^*-1}$, hence
\[
1\le\frac{|U|^{2^*-1}}{\sigma_1^{2^*-1}},\qquad\text{so}\qquad
M(1)\le\frac{M(1)}{\sigma_1^{2^*-1}}|U|^{2^*-1}.
\]
Inserting this into \eqref{est1-eps1} gives
\begin{equation}\label{sigma-2}
|\nabla F(x,U)|\le 2^*d(\eps)|U|^{2^*-1}\qquad\text{for }|U|\ge\sigma_1\text{ and a.e. }x\in\Omega,
\end{equation}
where $d(\eps):=1+\frac{M(1)}{2^*\sigma_1^{2^*-1}}$.

Combining the estimates for $|U|<\sigma_1$ in \eqref{sigma} and $|U|\ge\sigma_1$ in \eqref{sigma-2} gives \eqref{est3}; integrating \eqref{est3} gives \eqref{est4}.
\end{proof}

\begin{lemma}\label{lem:negative-quadratic}
Assume (F1)--(F2) and (F4), and let
$F^-(x,U):=\max\{-F(x,U),0\}$. Then
\begin{equation}\label{eq:negative-two-limits}
\frac{F^-(x,U)}{|U|^2}\longrightarrow0
\quad\text{as }|U|\to0\text{ and as }|U|\to\infty,
\end{equation}
uniformly in $x\in\Omega$. There exists $M>0$ such that
\begin{equation}\label{eq:negative-global-quadratic}
0\le F^-(x,U)\le M|U|^2
\qquad\text{for all }(x,U)\in\Omega\times\R^2.
\end{equation}
Moreover, for every $\theta>0$ there exist
$0<a_\theta<b_\theta$ and $M_\theta>0$ such that
\begin{equation}\label{eq:negative-amplitude-split}
F^-(x,U)\le\theta|U|^2+
M_\theta\mathbf1_{\{a_\theta\le|U|\le b_\theta\}},
\end{equation}
where $\mathbf1$ denotes the indicator function.
\end{lemma}

\begin{proof}
Fix $\theta>0$. By (F2), there exists $a_\theta>0$ such that
$|\nabla F(x,U)|\le2\theta|U|$ when $|U|\le a_\theta$,
uniformly in $x$. Using $F(x,0)=0$, we find
\[
|F(x,U)|
=\left|\int_0^1\nabla F(x,sU)\cdot U\dd s\right|
\le2\theta|U|^2\int_0^1s\dd s
=\theta|U|^2
\qquad (|U|\le a_\theta).
\]
Since $\theta$ is arbitrary, this proves the limit at zero.
By (F4), for $U\ne0$,
\[
0\le\frac{F^-(x,U)}{|U|^2}
\le C_1|U|^{q_1-2}.
\]
The right-hand side tends to zero at infinity because $q_1<2$.
Choose $b_\theta>a_\theta$ so large that
$C_1b_\theta^{q_1-2}\le\theta$. Then
$F^-(x,U)\le\theta|U|^2$ whenever $|U|\ge b_\theta$.
On the remaining annulus, (F4) gives
$F^-(x,U)\le C_1b_\theta^{q_1}$.
Thus \eqref{eq:negative-amplitude-split} holds with
$M_\theta=C_1b_\theta^{q_1}$.

To obtain a global quadratic bound, take $\theta=1$ in the
near-zero estimate. If $|U|\ge a_1$, then
\[
F^-(x,U)\le C_1|U|^{q_1-2}|U|^2
\le C_1a_1^{q_1-2}|U|^2.
\]
Hence \eqref{eq:negative-global-quadratic} holds with
$M=\max\{1,C_1a_1^{q_1-2}\}$.
\end{proof}

Recall from Definition \ref{def:W} the best Sobolev constant $S$. For bounded
$\Omega$ the infimum defining $S$ is not attained in $H_0^1(\Omega)$.
The corresponding quotient in the homogeneous Sobolev space on $\R^N$,
defined as the completion of $C_c^\infty(\R^N)$ in the gradient norm,
has the same infimum $S$, attained by the Aubin--Talenti functions (up to a normalization constant)
\cite{Aubin,Talenti}:
\begin{equation}\label{u11}
u_\eps(x)=\frac{1}{(\eps+|x|^2)^{(N-2)/2}},\qquad\eps>0.
\end{equation}

\begin{lemma}\label{lem:sys-sobolev}
For every $U\in W$,
\begin{equation}\label{sys-s}
\norm{U}_W^2\ge S\norm{U}_{L^{2^*}}^2,\qquad \|U\|_{L^{2^*}}:=\left(\int_\Omega|U|^{2^*}\dd x\right)^{1/2^*}.
\end{equation}
\end{lemma}

\begin{proof}
Write $U=(u_1,u_2)$ and let $|U|=(u_1^2+u_2^2)^{1/2}$. Since
$|\nabla|U||\le|\nabla U|$ a.e.\ in $\Omega$ (by the chain rule and the
Cauchy--Schwarz inequality) and $|U|=0$ on $\partial\Omega$ in the trace
sense, we have $|U|\in H_0^1(\Omega)$. Applying the scalar Sobolev inequality
to $|U|$,
\[
\|U\|_W^2=\int_\Omega|\nabla U|^2\dd x
\ge\int_\Omega|\nabla|U||^2\dd x
\ge S\|U\|_{L^{2^*}}^2.
\]
This gives \eqref{sys-s}.
\end{proof}

\begin{lemma}\label{lem:cc-system}
Let $\{U_j\}\subset W$ be a sequence with $\|U_j\|_W\le C$ satisfying
\begin{equation}\label{cc-hyp-b}
\|U_j\|_W^2-\int_\Omega|U_j|^{2^*}\dd x\to0.
\end{equation}
Then, up to a subsequence, either
\begin{itemize}
\item[(i)] $\int_\Omega|U_j|^{2^*}\dd x\to0$ and $\|U_j\|_W\to0$, or
\item[(ii)] $\int_\Omega|U_j|^{2^*}\dd x\to L\ge S^{N/2}$.
\end{itemize}
\end{lemma}

\begin{proof}
Since $\|U_j\|_W\le C$ and \eqref{cc-hyp-b} holds, the sequence $\int_\Omega|U_j|^{2^*}\dd x$ is bounded by Lemma \ref{lem:sys-sobolev}, hence (up to a subsequence) converges to some $L\ge0$.

If $L=0$, then $\int_\Omega|U_j|^{2^*}\dd x\to0$. By \eqref{cc-hyp-b},
\[
\|U_j\|_W^2=\int_\Omega|U_j|^{2^*}\dd x+o(1)\to0,
\]
so we are in case (i).

If $L>0$, by \eqref{cc-hyp-b} and \eqref{sys-s},
\[
L=\lim_j\|U_j\|_W^2\ge S\lim_j\|U_j\|_{L^{2^*}}^2=S L^{2/2^*}.
\]
Dividing by $L^{2/2^*}>0$, we get $L^{1-2/2^*}\ge S$, i.e.\ $L\ge S^{N/2}$. This is case (ii).
\end{proof}

Next, we prove the nonvanishing result used in the existence
argument: every $(PS)_c$ sequence with $0<c<S^{N/2}/N$ has a
subsequence converging weakly to a nontrivial critical point.

\begin{lemma}\label{lem:bdd-ae-weak}
Let $1<p<\infty$ and $m\ge1$. Suppose $\{f_j\}$ is bounded
in $L^p(\Omega;\R^m)$ and $f_j\to f$ a.e.\ in $\Omega$.
Then $f\in L^p(\Omega;\R^m)$ and
$f_j\rightharpoonup f$ weakly in $L^p(\Omega;\R^m)$.
\end{lemma}

\begin{proof} By Fatou's lemma, $f\in L^p(\Omega;\R^m)$.
Since $p\in(1,\infty)$, the space $L^p(\Omega;\R^m)$ is reflexive, and
$\{f_j\}$ is bounded in $L^p$. Hence there exist a subsequence
$\{f_{j_k}\}$ and $g\in L^p(\Omega;\R^m)$ such that
$f_{j_k}\rightharpoonup g$ weakly in $L^p$. Relabeling
$\{f_{j_k}\}$ as $\{f_j\}$, we have $f_j\rightharpoonup g$ weakly in
$L^p$. We claim that $g=f$.

Fix $\varphi\in C_c^\infty(\Omega;\R^m)$. Since $\Omega$ is bounded,
$C_c^\infty(\Omega;\R^m)\subset L^{p'}(\Omega;\R^m)$. By the
definition of weak convergence $f_j\rightharpoonup g$ in
$L^p(\Omega;\R^m)$,
\begin{equation}\label{gf1}
\int_\Omega g\cdot\varphi\dd x
=\lim_j\int_\Omega f_j\cdot\varphi\dd x.
\end{equation}
On the other hand, by the hypothesis of the lemma, $f_j\to f$ a.e.\
in $\Omega$, hence $f_j\cdot\varphi\to f\cdot\varphi$ a.e.\ in
$\Omega$. By Cauchy--Schwarz and the boundedness of $\varphi$,
\[
|f_j\cdot\varphi|\le|f_j|\,\|\varphi\|_{L^\infty}\quad\text{a.e.},
\]
and hence, by the pointwise estimate and the definition of the
$L^p$ norm,
\[
\|f_j\cdot\varphi\|_{L^p}
\le\|f_j\|_{L^p}\|\varphi\|_{L^\infty}
\le C\|\varphi\|_{L^\infty},
\]
where $C:=\sup_j\|f_j\|_{L^p}<\infty$. Thus $\{f_j\cdot\varphi\}$
is bounded in $L^p$. By H\"older's inequality, for every measurable $E\subset\Omega$,
\[
\int_E|f_j\cdot\varphi|\dd x
\le\|f_j\cdot\varphi\|_{L^p}|E|^{1/p'}
\le C\|\varphi\|_{L^\infty}|E|^{1/p'},
\]
where $C:=\sup_j\|f_j\|_{L^p}<\infty$. Taking $E=\Omega$ gives
$\sup_j\|f_j\cdot\varphi\|_{L^1}<\infty$ (uniform $L^1$-boundedness),
while the estimate with $|E|\to0$ gives uniform absolute continuity.
Hence $\{f_j\cdot\varphi\}$ is uniformly integrable. By Vitali's convergence theorem,
\begin{equation}\label{ff1}
\lim_j\int_\Omega f_j\cdot\varphi\dd x
=\int_\Omega f\cdot\varphi\dd x.
\end{equation}
Therefore, \eqref{gf1} and \eqref{ff1} imply that
$$
\int_\Omega (g-f)\cdot\varphi\dd x=0,
$$
for all
$\varphi\in C_c^\infty(\Omega;\R^m)$. Let $h:=g-f\in
L^p(\Omega;\R^m)$. Since $C_c^\infty(\Omega;\R^m)$ is dense in
$L^{p'}(\Omega;\R^m)$, for every $\psi\in L^{p'}(\Omega;\R^m)$ there
exists $\{\varphi_k\}\subset C_c^\infty(\Omega;\R^m)$ with
$\varphi_k\to\psi$ in $L^{p'}$, and hence
\[
\int_\Omega h\cdot\psi\dd x
=\lim_k\int_\Omega h\cdot\varphi_k\dd x=0.
\]
Taking $\psi=|h|^{p-2}h\in L^{p'}$ (when $h\ne0$) gives
$\int_\Omega|h|^p\dd x=0$, so $h=0$ a.e., i.e.\ $g=f$ in $L^p$.

Since the argument applies to any weakly convergent subsequence of
$\{f_j\}$, all such subsequences have the same weak limit $f$. By the
subsequence principle, $f_j\rightharpoonup f$ weakly in $L^p$.
\end{proof}

\begin{lemma}\label{PS-range}
Assume \eqref{posdef} and (F1)--(F3). Let $\{U_j\}\subset W$ satisfy
\[
\J(U_j)\to c,\qquad \norm{\J'(U_j)}_{W^*}\to0,
\qquad c\in\left(0,\frac{1}{N}S^{N/2}\right).
\]
Then $\{U_j\}$ is bounded in $W$, and up to a subsequence
$U_j\rightharpoonup U$ weakly in $W$, where $U\ne0$ is a critical point
of $\J$.
\end{lemma}

\begin{proof}
Throughout we use $N\ge3$ and $2^*=\frac{2N}{N-2}$.

\medskip
\textbf{Step 1: Boundedness.}
Since $\J(U_j)\to c$ and $\norm{\J'(U_j)}_{W^*}\to0$, the sequences
$\{\J(U_j)\}$ and $\{\norm{\J'(U_j)}_{W^*}\}$ are bounded; hence there
exists $\kappa>0$ such that
\begin{equation}\label{kappa}
|\J(U_j)|\le\kappa,\qquad
\left|\inner{\J'(U_j)}{U_j}\right|\le\kappa\norm{U_j}_W
\end{equation}
for all $j$. Indeed, the first inequality follows from the boundedness of
$\{\J(U_j)\}$; the second follows from the definition of the dual norm
$\norm{\J'(U_j)}_{W^*}=\sup_{\norm{V}_W\le1}|\inner{\J'(U_j)}{V}|$ and the
boundedness of $\{\norm{\J'(U_j)}_{W^*}\}$.
Using \eqref{J-intro} and \eqref{deriv},
\begin{equation}\label{identity}
\J(U_j)-\frac12\inner{\J'(U_j)}{U_j}
=\frac{1}{N}\int_\Omega |U_j|^{2^*}\dd x
+\int_\Omega\left(\frac12\nabla F(x,U_j)\cdot U_j-F(x,U_j)\right)\dd x.
\end{equation}
By Lemma \ref{lem:growth} with $\eps>0$ small,
\[
\left|\int_\Omega\left(\frac12\nabla F\cdot U_j-F\right)\dd x\right|
\le \frac{2^*+2}{2}\eps\int_\Omega|U_j|^{2^*}\dd x
+\frac{3}{2}M(\eps)\int_\Omega|U_j|\dd x.
\]
Inserting into \eqref{identity} and using \eqref{kappa},
\begin{equation}\label{bound-L2star}
\left(\frac{1}{N}-\frac{2^*+2}{2}\eps\right)\int_\Omega|U_j|^{2^*}\dd x
\le \kappa(1+\norm{U_j}_W)+\frac{3}{2}M(\eps)\int_\Omega|U_j|\dd x.
\end{equation}
Choose $\eps>0$ small enough that
$\frac{1}{N}-\frac{2^*+2}{2}\eps>0$. By the embedding
$W\hookrightarrow L^1\times L^1$ (recall Lemma \ref{rem:W-embeddings}),
\begin{equation}\label{UW-2}
\int_\Omega|U_j|\dd x\le C\|U_j\|_W.
\end{equation}
Combining with \eqref{bound-L2star},
\begin{equation}\label{bound-L2star-2}
\int_\Omega|U_j|^{2^*}\dd x\le \kappa_*(1+\norm{U_j}_W)
\end{equation}
for some $\kappa_*>0$.

On the other hand, by \eqref{coercive} and Lemma \ref{lem:growth} with
$\eps=1$,
\[
\J(U_j)
\ge \frac{c_0}{2}\norm{U_j}_W^2
-\left(\frac{1}{2^*}+1\right)\int_\Omega|U_j|^{2^*}\dd x
-M(1)\int_\Omega|U_j|\dd x.
\]
Combining with \eqref{kappa}, \eqref{UW-2} and \eqref{bound-L2star-2},
\[
\kappa\ge\frac{c_0}{2}\norm{U_j}_W^2
-\left(\left(\frac{1}{2^*}+1\right)\kappa_*+M(1)C\right)(1+\|U_j\|_W).
\]
Therefore
\[
\frac{c_0}{2}\norm{U_j}_W^2\le\kappa
+\left(\left(\frac{1}{2^*}+1\right)\kappa_*+M(1)C\right)(1+\|U_j\|_W).
\]
Hence $\{U_j\}$ is bounded in $W$.

\medskip
\textbf{Step 2: Weak convergence and $L^2$ compactness.}
Since $W$ is a Hilbert space, up to a subsequence
$U_j\rightharpoonup U$ weakly in $W$. Since the embedding
$W\hookrightarrow L^2\times L^2$ is compact
(recall Lemma \ref{rem:W-embeddings}), we also have
\begin{equation}\label{L2-strong}
U_j\to U\qquad\text{strongly in }L^2\times L^2.
\end{equation}
Passing to a further subsequence, we may assume
\begin{equation}\label{ae-conv}
U_j\to U\qquad\text{a.e.\ in }\Omega.
\end{equation}

Since $A^*$ is symmetric,
\[
A^*(U_j-U)\cdot(U_j+U)=A^*U_j\cdot U_j-A^*U\cdot U
\]
pointwise on $\Omega$. Hence
\[
\int_\Omega A^*U_j\cdot U_j\dd x-\int_\Omega A^*U\cdot U\dd x
=\int_\Omega A^*(U_j-U)\cdot(U_j+U)\dd x,
\]
and therefore
\[
\left|\int_\Omega A^*U_j\cdot U_j\dd x-\int_\Omega A^*U\cdot U\dd x\right|
\le C\|U_j-U\|_{L^2}\|U_j+U\|_{L^2},
\]
where $C:=\|A^*\|_{op}<\infty$. By \eqref{L2-strong} and the boundedness
of $\{U_j\}$ in $L^2\times L^2$, the right-hand side tends to $0$, giving
\begin{equation}\label{L2conv}
\int_\Omega A^*U_j\cdot U_j\dd x\to\int_\Omega A^*U\cdot U\dd x.
\end{equation}

\medskip
\textbf{Step 3: $U$ is a critical point.}
Fix $V\in W$. Since $\norm{\J'(U_j)}_{W^*}\to0$ and $V$ is fixed,
$\inner{\J'(U_j)}{V}\to0$. We pass to the limit in each term of
$\inner{\J'(U_j)}{V}$.

\emph{Linear term.}
By weak convergence $U_j\rightharpoonup U$ in $W$,
\begin{equation}\label{w-norm-11}
\inner{U_j}{V}_W\to\inner{U}{V}_W.
\end{equation}
Moreover, by \eqref{L2-strong},
\[
\int_\Omega A^*U_j\cdot V\dd x\to\int_\Omega A^*U\cdot V\dd x.
\]

\emph{Critical term.}
By \eqref{bound-L2star-2}, the sequence
$\{|U_j|^{2^*-2}U_j\}$ is bounded in $L^{2^*/(2^*-1)}$. By
\eqref{ae-conv},
\[
|U_j|^{2^*-2}U_j\to|U|^{2^*-2}U\qquad\text{a.e.\ in }\Omega.
\]
Thus, by Lemma \ref{lem:bdd-ae-weak},
\[
|U_j|^{2^*-2}U_j\rightharpoonup|U|^{2^*-2}U
\qquad\text{weakly in }L^{2^*/(2^*-1)}.
\]
Since $V\in W\hookrightarrow L^{2^*}$,
\begin{equation}\label{critical-term}
\int_\Omega|U_j|^{2^*-2}U_j\cdot V\dd x
\to\int_\Omega|U|^{2^*-2}U\cdot V\dd x.
\end{equation}

\emph{$F$ term.}
By Lemma \ref{lem:growth} with $\eps=1$, there exists $M(1)>0$ such that
\[
|\nabla F(x,U)|\le 2^*|U|^{2^*-1}+M(1)\qquad\text{a.e.\ in }\Omega.
\]
Let $p:=2^*/(2^*-1)$. By $(a+b)^p\le 2^{p-1}(a^p+b^p)$ and
$(2^*-1)p=2^*$,
\[
|\nabla F(x,U_j)|^p
\le 2^{p-1}\left((2^*)^p|U_j|^{2^*}+M(1)^p\right)\quad\text{a.e.}
\]
Integrating and using \eqref{bound-L2star-2} together with the boundedness
of $\{U_j\}$ in $W$, we obtain
\[
\int_\Omega|\nabla F(x,U_j)|^p\dd x\le C
\]
for a constant $C$ independent of $j$. Hence $\{\nabla F(\cdot,U_j)\}$ is bounded in
$L^{2^*/(2^*-1)}$.

Since $F\in C^1$ and $U_j\to U$ a.e.\ in $\Omega$ (by
\eqref{ae-conv}), we have
$$
\nabla F(x,U_j)\to\nabla F(x,U) \ \mbox{a.e. in} \  \Omega.
$$
Applying
Lemma \ref{lem:bdd-ae-weak} with $p=2^*/(2^*-1)$,
\[
\nabla F(\cdot,U_j)\rightharpoonup\nabla F(\cdot,U)
\qquad\text{weakly in }L^{2^*/(2^*-1)}.
\]
Since $V\in L^{2^*}$,
\begin{equation}\label{F-term}
\int_\Omega\nabla F(x,U_j)\cdot V\dd x
\to\int_\Omega\nabla F(x,U)\cdot V\dd x.
\end{equation}

Combining these limits with $\inner{\J'(U_j)}{V}\to0$, we get
$$
\inner{\J'(U)}{V}=0,\quad \forall V\in W,
$$
so $U$ is a critical
point of $\J$.

\medskip
\textbf{Step 4: $U\ne0$.}
Suppose, by contradiction, that $U=0$.
We claim that
\begin{equation}\label{vanishing}
\int_\Omega F(x,U_j)\dd x\to0,\qquad
\int_\Omega \nabla F(x,U_j)\cdot U_j\dd x\to0.
\end{equation}
Indeed, by Lemma \ref{lem:growth},
\[
\left|\int_\Omega F(x,U_j)\dd x\right|
\le\eps\int_\Omega|U_j|^{2^*}\dd x+M(\eps)\int_\Omega|U_j|\dd x.
\]
By the boundedness already proved, set
$M_*:=\sup_j\int_\Omega|U_j|^{2^*}\dd x<\infty$.
The first term is bounded by $\eps M_*$, while the second tends to $0$ because
$U_j\to0$ in $L^1$, by $U=0$ and the compact embedding
$W\hookrightarrow L^1\times L^1$
in Lemma \ref{rem:W-embeddings}. Hence
\[
\limsup_{j\to\infty}\left|\int_\Omega F(x,U_j)\dd x\right|
\le\eps M_*.
\]
Letting $\eps\to0^+$ gives $\int_\Omega F(x,U_j)\dd x\to0$. Similarly,
\[
|\nabla F(x,U_j)\cdot U_j|
\le 2^*\eps|U_j|^{2^*}+M(\eps)|U_j|,
\]
so the same argument gives
$\int_\Omega\nabla F(x,U_j)\cdot U_j\dd x\to0$. This proves
\eqref{vanishing}.

By \eqref{deriv},
\[
\inner{\J'(U_j)}{U_j}
=\norm{U_j}_W^2-\int_\Omega A^*U_j\cdot U_j\dd x
-\int_\Omega|U_j|^{2^*}\dd x
-\int_\Omega\nabla F(x,U_j)\cdot U_j\dd x.
\]
It follows from $\inner{\J'(U_j)}{U_j}=o(1)$,
\eqref{L2conv} with $U=0$, and
\eqref{vanishing} that
\begin{equation}\label{norm-eq}
\norm{U_j}_W^2-\int_\Omega|U_j|^{2^*}\dd x=o(1).
\end{equation}
By \eqref{norm-eq} and the boundedness of $\{U_j\}$ in $W$, we may
extract a further subsequence such that
\begin{equation}\label{L-def}
\norm{U_j}_W^2\to L,\qquad \int_\Omega|U_j|^{2^*}\dd x\to L
\end{equation}
for one and the same $L\ge0$. From $\J(U_j)\to c$, \eqref{L2conv} with $U=0$, and \eqref{vanishing},
\begin{equation}\label{c-L}
c=\left(\frac12-\frac{1}{2^*}\right)L=\frac{1}{N}L.
\end{equation}
Since $c>0$ by assumption, \eqref{c-L} gives $L=Nc>0$.

Applying Lemma \ref{lem:cc-system} with \eqref{norm-eq} and \eqref{L-def},
we get $L\ge S^{N/2}$. Therefore, by \eqref{c-L},
\[
c=\frac{1}{N}L\ge\frac{1}{N}S^{N/2},
\]
which contradicts the hypothesis $c<\frac{1}{N}S^{N/2}$. Hence $U\ne0$.

\medskip
Combining Steps 3 and 4, $U\ne0$ is a critical point of $\J$, hence a
nontrivial weak solution of \eqref{sys1}.
\end{proof}

\begin{lemma}\label{mp-geometry}
Assume \eqref{posdef} and (F1)--(F3). Then there exist $\rho>0$ and $\beta>0$ such that
\[
\J(U)\ge\beta\quad\text{for all }U\in W\text{ with }\norm{U}_W=\rho.
\]
Moreover, for any $U\in W\setminus\{0\}$, $\J(tU)\to-\infty$ as $t\to+\infty$.
\end{lemma}

\begin{proof} By Lemma \ref{lem:growth2}, for every $\eps>0$ there exists $d(\eps)>0$ with
\[
|F(x,U)|\le\eps|U|^2+d(\eps)|U|^{2^*}.
\]
Using \eqref{coercive} and the Sobolev embeddings $W\hookrightarrow L^2\times L^2$ (constant $C_2$) and $W\hookrightarrow L^{2^*}\times L^{2^*}$ (constant $C_{2^*}$),
\begin{align*}
\J(U)&\ge\frac12 c_0\norm{U}_W^2-\left(\frac{1}{2^*}+d(\eps)\right)C_{2^*}^{2^*}\norm{U}_W^{2^*}-\eps C_2^2\norm{U}_W^2\\
&=\left(\frac{c_0}{2}-\eps C_2^2\right)\norm{U}_W^2-\left(\frac{1}{2^*}+d(\eps)\right)C_{2^*}^{2^*}\norm{U}_W^{2^*}.
\end{align*}
Choose $\eps>0$ small enough that $\alpha:=\frac{c_0}{2}-\eps C_2^2>0$. Then for $\norm{U}_W=\rho$ small enough that
$$
\alpha-\left(\frac{1}{2^*}+d(\eps)\right)C_{2^*}^{2^*}\rho^{2^*-2}>0,
$$
we obtain
$$
\J(U)\ge\beta:=
\rho^2\left(\alpha-\left(\frac{1}{2^*}+d(\eps)\right)C_{2^*}^{2^*}\rho^{2^*-2}\right)>0.
$$

For $U$ nonzero, using \eqref{est2} in Lemma \ref{lem:growth},
\[
\J(tU)\le \frac{t^2}{2}\left(\norm{U}_W^2-\int_\Omega A^*U\cdot U\dd x\right)
-\left(\frac{1}{2^*}-\eps\right)t^{2^*}\int_\Omega|U|^{2^*}\dd x+M(\eps)t\int_\Omega|U|\dd x.
\]
Choose $\eps>0$ small enough that $\frac{1}{2^*}-\eps>0$. Since $2^*>2$ and the coefficient of $t^{2^*}$ is negative, $\J(tU)\to-\infty$ as $t\to+\infty$.
\end{proof}

\section{The abstract result: proof of Theorem \ref{thm:abstract}}\label{sec:abstract}

\begin{proof}[{\bf Proof of Theorem \ref{thm:abstract}}]
Assume the energy threshold condition \eqref{energy-cond} holds for some $U_0\in W\setminus\{0\}$. By Lemma \ref{mp-geometry}, $\J$ has the mountain pass geometry. Let $e=t_0 U_0$ with $t_0$ large enough that $\J(e)<0$. Define
\[
c=\inf_{\gamma\in\Gamma}\max_{t\in[0,1]}\J(\gamma(t)),\qquad
\Gamma=\{\gamma\in C([0,1],W):\gamma(0)=0,\ \gamma(1)=e\}.
\]

\medskip
\textbf{Claim 1: $c<\frac{1}{N}S^{N/2}$.}
The path $t\mapsto t t_0 U_0$ belongs to $\Gamma$, so \eqref{energy-cond} implies
\[
c\le\max_{t\in[0,1]}\J(t t_0 U_0)\le\sup_{t\ge0}\J(tU_0)<\frac{1}{N}S^{N/2}.
\]

\medskip
\textbf{Claim 2: $c\ge\beta>0$.}
For any $\gamma\in\Gamma$, the function $t\mapsto\norm{\gamma(t)}_W$ is continuous, $\norm{\gamma(0)}_W=0<\rho$, and $\norm{\gamma(1)}_W=\norm{e}_W>\rho$ (choosing $t_0$ large enough). Hence there exists $\bar t\in(0,1)$ with $\norm{\gamma(\bar t)}_W=\rho$. By Lemma \ref{mp-geometry},
\[
\max_{t\in[0,1]}\J(\gamma(t))\ge\J(\gamma(\bar t))\ge\beta.
\]
Taking the infimum over $\gamma\in\Gamma$ gives $c\ge\beta>0$.

\medskip
By a variant of the mountain pass theorem without the Palais--Smale condition (see \cite{AmbrosettiRabinowitz,BrezisNirenberg}), there exists a sequence $\{U_j\}\subset W$ such that
\[
\J(U_j)\to c,\qquad \norm{\J'(U_j)}_{W^*}\to0.
\]
By Claim 1 and Claim 2, we have
$$
0<c<\frac{1}{N}S^{N/2}.
$$
Hence Lemma \ref{PS-range} applies and gives a nontrivial critical
point $U\in W$ of $\J$. In particular, \eqref{sys1} admits a nontrivial weak solution
$U\in W$.
\end{proof}

\section{The case \texorpdfstring{$N\ge4$}{N >= 4}: proof of Theorem
\ref{thm:N4}}\label{sec:N4}

After a translation, assume $0\in\Omega$. Let $u_\eps$ be the
Aubin--Talenti function given in \eqref{u11}, and choose
$\eta\in C_c^\infty(\Omega)$ with $0\le\eta\le1$ and
$\eta\equiv1$ near $0$. Normalize
\begin{equation}\label{def-v-eps}
v_\eps:=\frac{w_\eps}{\|w_\eps\|_{L^{2^*}(\Omega)}},\quad
w_\eps:=\eta u_\eps,
\end{equation}
so that $\|v_\eps\|_{L^{2^*}(\Omega)}=1$.

\begin{remark}
\label{rem:reduction-symmetry}
Let $\mathbf e=(e_1,e_2)\in\R^2$ be a unit eigenvector of
$A^*$ corresponding to $\xi>0$, so that
\[
A^*\mathbf e=\xi\mathbf e,\qquad e_1^2+e_2^2=1.
\]
We use the test function
\[
U_\eps:=v_\eps\mathbf e=(e_1v_\eps,e_2v_\eps)\in W.
\]
It is nonzero, and
\[
|U_\eps|^2=(e_1^2+e_2^2)|v_\eps|^2=|v_\eps|^2,
\qquad
|\nabla U_\eps|^2
=(e_1^2+e_2^2)|\nabla v_\eps|^2=|\nabla v_\eps|^2.
\]
Moreover,
\[
\int_\Omega A^*U_\eps\cdot U_\eps\dd x
=(A^*\mathbf e\cdot\mathbf e)\int_\Omega|v_\eps|^2\dd x
=\xi\int_\Omega|v_\eps|^2\dd x.
\]
Since the assumptions on $F$ hold for all $U\in\R^2$, they
apply along $t v_\eps\mathbf e$ for either sign of the
components of $\mathbf e$. No rotational symmetry of $F$ is
needed. This choice is used only to construct a test ray; the
mountain pass argument remains in the full space $W$.
\end{remark}

\begin{lemma}\label{lem:classical}
Let $v_\eps$ be as in \eqref{def-v-eps}. Then, as $\eps\to0^+$,
\begin{align}
\int_\Omega|\nabla v_\eps|^2\dd x&=S+O(\eps^{(N-2)/2}),
\label{est-grad}\\
\int_\Omega|v_\eps|^{2^*}\dd x&=1,\label{est-L2star}\\
\int_\Omega|v_\eps|^2\dd x&=
\begin{cases}
C\eps+O(\eps^{(N-2)/2}), & N>4,\\[1mm]
C\eps|\log\eps|+O(\eps), & N=4,
\end{cases}\label{est-L2}
\end{align}
where $C>0$ depends only on $N$.
\end{lemma}

\begin{proof}
By the classical estimates \cite[Lemma 1.1]{BrezisNirenberg}, as
$\eps\to0^+$,
\begin{align}
\int_\Omega|\nabla w_\eps|^2\dd x&=\frac{K_1}{\eps^{(N-2)/2}}+O(1),
\label{est-grad-raw}\\
\|w_\eps\|_{L^{2^*}(\Omega)}^2&=\frac{K_2}{\eps^{(N-2)/2}}+O(\eps),
\label{est-L2star-raw}\\
\int_\Omega|w_\eps|^2\dd x&=
\begin{cases}
\dfrac{K_3}{\eps^{(N-4)/2}}+O(1), & N>4,\\[2mm]
K_3|\log\eps|+O(1), & N=4,
\end{cases}
\label{est-L2-raw}
\end{align}
where $K_1,K_2,K_3>0$ depend only on $N$ and $K_1/K_2=S$.

\emph{Verification of \eqref{est-grad}.}
Since $v_\eps=w_\eps/\|w_\eps\|_{L^{2^*}(\Omega)}$,
\[
\int_\Omega|\nabla v_\eps|^2\dd x
=\frac{\int_\Omega|\nabla w_\eps|^2\dd x}
{\|w_\eps\|_{L^{2^*}(\Omega)}^2}
=\frac{K_1\eps^{-(N-2)/2}+O(1)}
{K_2\eps^{-(N-2)/2}+O(\eps)}.
\]
Factoring out $K_j\eps^{-(N-2)/2}$ from the numerator and denominator,
\[
=\frac{K_1\eps^{-(N-2)/2}\left(1+O(\eps^{(N-2)/2})\right)}
{K_2\eps^{-(N-2)/2}\left(1+O(\eps^{N/2})\right)}
=\frac{K_1}{K_2}\cdot\frac{1+O(\eps^{(N-2)/2})}
{1+O(\eps^{N/2})}.
\]
Using $(1+O(\eps^{N/2}))^{-1}=1+O(\eps^{N/2})$ and $K_1/K_2=S$,
\[
\int_\Omega|\nabla v_\eps|^2\dd x
=S\left(1+O(\eps^{(N-2)/2})\right)\left(1+O(\eps^{N/2})\right)
=S+O(\eps^{(N-2)/2}),
\]
where the last step uses $(N-2)/2<N/2$.

\emph{Verification of \eqref{est-L2star}.}
This is the normalization condition \eqref{def-v-eps}:
\[
\int_\Omega|v_\eps|^{2^*}\dd x=1.
\]

\emph{Verification of \eqref{est-L2}.}
Again by \eqref{def-v-eps},
\[
\int_\Omega|v_\eps|^2\dd x
=\frac{\int_\Omega|w_\eps|^2\dd x}
{\|w_\eps\|_{L^{2^*}(\Omega)}^2}.
\]
For $N>4$, using \eqref{est-L2-raw} and \eqref{est-L2star-raw},
\[
\int_\Omega|v_\eps|^2\dd x
=\frac{K_3\eps^{-(N-4)/2}+O(1)}
{K_2\eps^{-(N-2)/2}+O(\eps)}
=\frac{K_3}{K_2}\eps\cdot\frac{1+O(\eps^{(N-4)/2})}
{1+O(\eps^{N/2})}
=C\eps+O(\eps^{(N-2)/2}),
\]
where $C:=K_3/K_2>0$. For $N=4$,
\begin{equation}\label{eq:N4-L2}
\begin{aligned}
\int_\Omega|v_\eps|^2\,dx
=\frac{K_3|\log\eps|+O(1)}{K_2\eps^{-1}+O(\eps)}
&=\frac{K_3|\log\eps|\left(1+O\left(\frac1{|\log\eps|}\right)\right)}
{K_2\eps^{-1}\left(1+O(\eps^2)\right)}\\
&=C\eps|\log\eps|\,
\frac{1+O\left(\frac1{|\log\eps|}\right)}{1+O(\eps^2)}\\
&=C\eps|\log\eps|\left(1+O\left(\frac1{|\log\eps|}\right)\right)
\left(1+O(\eps^2)\right)\\
&=C\eps|\log\eps|+O(\eps),
\end{aligned}
\end{equation}
where $C:=\frac{K_3}{K_2}>0$.
\end{proof}

To prove Theorem \ref{thm:N4}, we first establish the following estimate for $N>4$, then apply it to the negative part of $F$.

\begin{lemma}\label{lem:Lq}
Let $N>4$, let $v_\eps$ be as in \eqref{def-v-eps}, and let $q\in\left(\dfrac{4}{N-2},2\right)$.
Then, as $\eps\to0^+$,
\begin{equation}\label{eq:Lq-q1}
\int_\Omega|v_\eps|^{q}\dd x=o(\eps).
\end{equation}
\end{lemma}
\begin{proof} Fix $r_0>0$ so small that $B_{2r_0}(0)\subset\subset\Omega$ and
$\eta\equiv1$ on $B_{r_0}(0)$. Splitting the integral,
\begin{equation}\label{eq:Lq-split}
\int_\Omega|w_\eps|^{q}\,dx
=\underbrace{\int_{B_{r_0}(0)}|u_\eps|^{q}\,dx}_{=:I_1}
+\underbrace{\int_{\Omega\setminus B_{r_0}(0)}|\eta u_\eps|^{q}\,dx}_{=:I_2}.
\end{equation}

\emph{Estimate of $I_1$.}
By polar coordinates $x=r\omega$, $r=|x|$,
$dx=r^{N-1}dr\,d\omega$, where $\omega_N$ is the area of $S^{N-1}$,
\[
I_1=\omega_N\int_0^{r_0}
(\eps+r^2)^{-\frac{(N-2)q}{2}}r^{N-1}\,dr.
\]
Perform the change of variables $r=\eps^{1/2}s$, $dr=\eps^{1/2}ds$,
so that $r=r_0\Leftrightarrow s=r_0\eps^{-1/2}$, and
\[
r^{N-1}=\eps^{\frac{N-1}{2}}s^{N-1},\qquad
\eps+r^2=\eps(1+s^2).
\]
Therefore
\[
(\eps+r^2)^{-\frac{(N-2)q}{2}}
=\eps^{-\frac{(N-2)q}{2}}(1+s^2)^{-\frac{(N-2)q}{2}},
\]
and hence
\begin{equation}\label{eq:I1-main}
I_1=\omega_N\eps^{\frac N2-\frac{(N-2)q}{2}}
\int_0^{r_0\eps^{-1/2}}
(1+s^2)^{-\frac{(N-2)q}{2}}s^{N-1}\,ds.
\end{equation}

\emph{Estimate of $I_2$.}
Since $|x|\ge r_0$ on $\Omega\setminus B_{r_0}(0)$,
\[
u_\eps(x)=(\eps+|x|^2)^{-\frac{N-2}{2}}
\le(r_0^2)^{-\frac{N-2}{2}}=r_0^{-(N-2)}.
\]
Since $\eta\in L^\infty(\Omega)$,
\[
|\eta(x)u_\eps(x)|^{q}
\le\|\eta\|_{L^\infty}^{q}r_0^{-(N-2)q},
\]
and therefore
\begin{equation}\label{eq:I2-main}
I_2\le\|\eta\|_{L^\infty}^{q}r_0^{-(N-2)q}|\Omega|=O(1).
\end{equation}

By \eqref{est-L2star-raw},
\[
\|w_\eps\|_{L^{2^*}(\Omega)}^2
=K_2\,\eps^{-\frac{N-2}{2}}\bigl(1+O(\eps^{N/2})\bigr),
\]
and raising both sides to the power $q/2$ gives
\begin{equation}\label{eq:norm-w-q1}
\|w_\eps\|_{L^{2^*}(\Omega)}^{q}
=K_2^{q/2}\,\eps^{-\frac{(N-2)q}{4}}
\bigl(1+O(\eps^{N/2})\bigr),
\end{equation}
where we used the Taylor expansion
$(1+u)^{q/2}=1+\frac{q}{2}u+O(u^2)$ as $u\to0$, applied to
$u=O(\eps^{N/2})$, together with $O(u^2)=O(\eps^N)\subset O(\eps^{N/2})$.

Note that $q\in\left(\dfrac{4}{N-2},2\right)$ and
$\dfrac{4}{N-2}<\dfrac{N}{N-2}<2$. We distinguish three cases.

\smallskip

\textbf{Case 1: $\dfrac{N}{N-2}<q<2$.}
Then the integrand in \eqref{eq:I1-main} behaves like
$s^{-(N-2)q+N-1}$ as $s\to+\infty$, with exponent
\begin{equation}\label{eq:qn-main}
-(N-2)q+N-1<-1
\iff q>\frac{N}{N-2}.
\end{equation}
Hence the integral converges, and setting
\[
K_{q}:=\omega_N\int_0^\infty
(1+s^2)^{-\frac{(N-2)q}{2}}s^{N-1}\,ds>0,
\]
we obtain, for $R=r_0\eps^{-1/2}$,
\[
\int_0^{r_0\eps^{-1/2}}
(1+s^2)^{-\frac{(N-2)q}{2}}s^{N-1}\,ds
=\frac{K_{q}}{\omega_N}+o(1).
\]
Inserting this into \eqref{eq:I1-main} gives
\begin{equation}\label{eq:norm-I1}
I_1=K_{q}\,\eps^{\frac N2-\frac{(N-2)q}{2}}
+o\!\left(\eps^{\frac N2-\frac{(N-2)q}{2}}\right).
\end{equation}
Hence, \eqref{eq:Lq-split} and \eqref{eq:I2-main}--\eqref{eq:norm-I1} imply that
\begin{equation}\label{eq:norm-w-q0}
\int_\Omega|w_\eps|^{q}\,dx=I_1+I_2
=K_{q}\,\eps^{\frac N2-\frac{(N-2)q}{2}}
+o\!\left(\eps^{\frac N2-\frac{(N-2)q}{2}}\right),
\end{equation}
where we used $\frac N2-\frac{(N-2)q}{2}<0$. Combining \eqref{eq:norm-w-q0} and \eqref{eq:norm-w-q1},
\[
\int_\Omega|v_\eps|^{q}\,dx
=\frac{K_{q}}{K_{2}^{q/2}}
\,\eps^{\gamma(q)}\bigl(1+o(1)\bigr),
\qquad
\gamma(q)=\frac{2N-(N-2)q}{4}.
\]
Since $q<2$,
\[
\gamma(q)>\frac{2N-(N-2)\cdot2}{4}=1,
\]
so $\int_\Omega|v_\eps|^{q}\,dx=o(\eps)$.

\smallskip

\textbf{Case 2: $q=\dfrac{N}{N-2}$.}
Then $(N-2)q=N$, so the integrand in \eqref{eq:I1-main} is
$(1+s^2)^{-N/2}s^{N-1}$. For $s\ge1$, we have
\[
s^{-1}\lesssim(1+s^2)^{-\frac N2}s^{N-1}\lesssim s^{-1}
\qquad(s\ge1).
\]
Splitting the integral at $s=1$,
\[
\int_0^{r_0\eps^{-1/2}}(1+s^2)^{-\frac N2}s^{N-1}\,ds
=\int_0^1(\cdots)\dd s+\int_1^{r_0\eps^{-1/2}}(\cdots)\dd s.
\]
The first integral is bounded by $1$, since $(1+s^2)^{-N/2}s^{N-1}
\le 1$ for $0<s\le1$. For $s>1$, we have $s^2\le 1+s^2\le 2s^2$,
hence
\[
2^{-N/2}s^{-1}\le(1+s^2)^{-\frac N2}s^{N-1}\le s^{-1}.
\]
Therefore the second integral satisfies
\[
\int_1^{r_0\eps^{-1/2}}(1+s^2)^{-\frac N2}s^{N-1}\dd s
\le\int_1^{r_0\eps^{-1/2}}\frac{\dd s}{s}
=\frac12|\log\eps|+\log r_0=O(|\log\eps|).
\]
Combining the two estimates,
\[
\int_0^{r_0\eps^{-1/2}}(1+s^2)^{-\frac N2}s^{N-1}\,ds
=O(|\log\eps|).
\]
Inserting this into \eqref{eq:I1-main} and using $q=\dfrac{N}{N-2}$ and
$\frac N2-\frac{(N-2)q}{2}=0$, we obtain
\[
I_1=O(|\log\eps|).
\]
It follows from \eqref{eq:Lq-split} and \eqref{eq:I2-main} that
\[
\int_\Omega|w_\eps|^{q}\,dx=I_1+I_2=O(|\log\eps|).
\]
Dividing by \eqref{eq:norm-w-q1} and using $q=N/(N-2)$,
\[
\int_\Omega|v_\eps|^{q}\,dx
=O\!\left(\eps^{N/4}|\log\eps|\right).
\]
Since $N>4$, we have $N/4>1$, and therefore
$\eps^{N/4}|\log\eps|=o(\eps)$. Hence
$$\int_\Omega|v_\eps|^{q}\,dx=o(\eps).
$$

\smallskip

\textbf{Case 3: $\dfrac{4}{N-2}<q<\dfrac{N}{N-2}$.}
Then $(N-2)q<N$, so the integrand in \eqref{eq:I1-main} behaves
like $s^{-(N-2)q+N-1}$ as $s\to+\infty$, with exponent
\[
-(N-2)q+N-1>-1
\iff q<\frac{N}{N-2}.
\]
Hence the integral $\int_1^\infty s^{-(N-2)q+N-1}\dd s$ diverges
polynomially. Splitting at $s=1$ and integrating,
\[
\int_1^{r_0\eps^{-1/2}}s^{-(N-2)q+N-1}\dd s
=\frac{(r_0\eps^{-1/2})^{\,N-(N-2)q}-1}{N-(N-2)q}
\sim\frac{r_0^{\,N-(N-2)q}}{N-(N-2)q}
\eps^{\frac{(N-2)q-N}{2}},
\]
where we used $(N-2)q-N<0$ (because
$q<N/(N-2)$). The integral over $0<s\le1$ is $O(1)$ and is of
lower order (since $(N-2)q-N<0$, so the main term diverges as
$\eps\to0^+$). Hence
\[
\int_0^{r_0\eps^{-1/2}}
(1+s^2)^{-\frac{(N-2)q}{2}}s^{N-1}\,ds
\sim C\eps^{\frac{(N-2)q-N}{2}},
\]
with $C:=\dfrac{r_0^{\,N-(N-2)q}}{N-(N-2)q}>0$.
Inserting this into \eqref{eq:I1-main} yields
\[
I_1\sim \omega_N C\eps^{\frac N2-\frac{(N-2)q}{2}}
\cdot\eps^{\frac{(N-2)q-N}{2}}
=\omega_N C.
\]
Hence $I_1=O(1)$. It follows from \eqref{eq:Lq-split} and \eqref{eq:I2-main} that
\[
\int_\Omega|w_\eps|^{q}\,dx=I_1+I_2=O(1).
\]
Dividing by \eqref{eq:norm-w-q1},
\[
\int_\Omega|v_\eps|^{q}\,dx
=O\!\left(\eps^{\frac{(N-2)q}{4}}\right).
\]
Since $q>\dfrac{4}{N-2}$ by the choice of $q$, we have
$\frac{(N-2)q}{4}>1$, and therefore
\[
\int_\Omega|v_\eps|^{q}\,dx
=O\!\left(\eps^{\frac{(N-2)q}{4}}\right)=o(\eps).
\]

Therefore, in all cases we obtain \eqref{eq:Lq-q1}.
\end{proof}

\begin{lemma}\label{lem:negative-potential}
Let $N>4$, let $v_\eps$ be as in \eqref{def-v-eps}, and assume
(F1)--(F2) and (F4). For every unit vector $\mathbf e\in\R^2$
and every $T>0$,
\begin{equation}\label{eq:negative-potential}
\sup_{0\le t\le T}\int_\Omega F^-(x,tv_\eps\mathbf e)\dd x
=o(\eps),
\end{equation}
where $F^-(x,U):=\max\{-F(x,U),0\}$.
\end{lemma}

\begin{proof}
By (F2), there exists $\rho\in(0,1)$ such that
\[
|\nabla F(x,U)|\le|U|\qquad\text{for }|U|\le\rho,
\]
uniformly in $x$. Since $\nabla F$ is continuous on the compact set
$\overline\Omega\times\{U\in\R^2:|U|\le1\}$,
there exists $M_1>0$ such that
$|\nabla F(x,U)|\le M_1$ for $|U|\le1$. Hence, when
$\rho<|U|\le1$,
\[
|\nabla F(x,U)|\le M_1\le\frac{M_1}{\rho}|U|.
\]
With $M_2:=\max\{1,M_1/\rho\}$, the two estimates give
$|\nabla F(x,U)|\le M_2|U|$ for $|U|\le1$.
Since $F(x,0)=0$, integration along the segment from $0$ to $U$
yields
\[
|F(x,U)|
=\left|\int_0^1\nabla F(x,sU)\cdot U\dd s\right|
\le M_2|U|^2\int_0^1s\dd s
=\frac{M_2}{2}|U|^2
\qquad\text{for }|U|\le1.
\]
Choose
\[
q\in\left(\max\left\{q_1,\frac4{N-2}\right\},2\right),
\qquad C_F:=\max\left\{\frac{M_2}{2},C_1\right\}.
\]
For $|U|\le1$, the inequality $q<2$ gives
\[
F^-(x,U)\le\frac{M_2}{2}|U|^2\le C_F|U|^q.
\]
For $|U|\ge1$, (F4) and $q>q_1$ give
\[
F^-(x,U)\le C_1|U|^{q_1}\le C_F|U|^q.
\]
Consequently,
\begin{equation}\label{eq:negative-power}
0\le F^-(x,U)\le C_F|U|^q
\qquad\text{for all }(x,U)\in\Omega\times\R^2.
\end{equation}
Since $|\mathbf e|=1$, for $0\le t\le T$ we obtain
\[
0\le\int_\Omega F^-(x,tv_\eps\mathbf e)\dd x
\le C_Ft^q\int_\Omega|v_\eps|^q\dd x
\le C_FT^q\int_\Omega|v_\eps|^q\dd x.
\]
Lemma \ref{lem:Lq} applies to $q$ and makes the last expression
$o(\eps)$, independently of $t\in[0,T]$. This proves
\eqref{eq:negative-potential}.
\end{proof}

\begin{lemma}\label{lem:negative-N4}
Let $N=4$, let $v_\eps$ be as in \eqref{def-v-eps}, and
assume (F1)--(F2) and (F4). For every unit vector
$\mathbf e\in\R^2$ and every $T>0$,
\begin{equation}\label{eq:negative-N4}
\sup_{0\le t\le T}
\int_\Omega F^-(x,tv_\eps\mathbf e)\dd x
=o(\eps|\log\eps|).
\end{equation}
\end{lemma}

\begin{proof}
By \eqref{est-L2star-raw} with $N=4$,
$\|w_\eps\|_4\ge c\eps^{-1/2}$ for all sufficiently small
$\eps$. Since $0\le\eta\le1$, there exists $C>0$ such that
\[
0\le v_\eps(x)
=\frac{\eta(x)}{\|w_\eps\|_4(\eps+|x|^2)}
\le\frac{C\eps^{1/2}}{\eps+|x|^2}.
\]
Fix $\theta>0$, and take $a_\theta,b_\theta,M_\theta$
from Lemma \ref{lem:negative-quadratic}. If $0<t\le T$ and
$tv_\eps(x)\ge a_\theta$, then
\[
|x|^2\le\eps+|x|^2
\le\frac{CT}{a_\theta}\eps^{1/2}.
\]
The set $\{x:tv_\eps(x)\ge a_\theta\}$ is therefore contained
in a ball of radius $(CT/a_\theta)^{1/2}\eps^{1/4}$.
Its four-dimensional measure is bounded by
\[
\bigl|\{x\in\Omega:tv_\eps(x)\ge a_\theta\}\bigr|
\le C_{\theta,T}\eps,
\]
uniformly for $0\le t\le T$; the set is empty when $t=0$.
Using \eqref{eq:negative-amplitude-split} and $|\mathbf e|=1$,
we obtain
\begin{align*}
\int_\Omega F^-(x,tv_\eps\mathbf e)\dd x
&\le\theta t^2\int_\Omega v_\eps^2\dd x
+M_\theta\bigl|\{x:a_\theta\le tv_\eps(x)\le b_\theta\}\bigr|\\
&\le\theta T^2\int_\Omega v_\eps^2\dd x
+M_\theta C_{\theta,T}\eps.
\end{align*}
By \eqref{est-L2}, there is a constant $C_L>0$, independent
of $\theta$ and $\eps$, such that
$\int_\Omega v_\eps^2\dd x\le C_L\eps|\log\eps|$ for
small $\eps$. Consequently,
\[
\limsup_{\eps\to0^+}
\frac{\displaystyle\sup_{0\le t\le T}
\int_\Omega F^-(x,tv_\eps\mathbf e)\dd x}
{\eps|\log\eps|}
\le C_LT^2\theta.
\]
Letting $\theta\to0^+$ proves \eqref{eq:negative-N4}.
\end{proof}

\begin{proof}[{\bf Proof of Theorem \ref{thm:N4}}]
Let $U_\eps=v_\eps\mathbf e$, and set
\[
A_\eps:=\int_\Omega|\nabla v_\eps|^2\dd x
-\xi\int_\Omega|v_\eps|^2\dd x.
\]
Recall from \eqref{def-v-eps} and Remark \ref{rem:reduction-symmetry}
that $\int_\Omega|v_\eps|^{2^*}\dd x=1$ and
$\int_\Omega A^*U_\eps\cdot U_\eps\dd x
=\xi\int_\Omega|v_\eps|^2\dd x$.
By coercivity, $A_\eps>0$. For $t\ge0$,
\[
\J(tU_\eps)=\frac{t^2}{2}A_\eps-\frac{t^{2^*}}{2^*}
-\int_\Omega F(x,tv_\eps\mathbf e)\dd x.
\]
By Lemma \ref{lem:classical},
\begin{equation}\label{AS-01}
A_\eps=
\begin{cases}
S-\xi C\eps+O(\eps^{(N-2)/2}), & N>4,\\[1mm]
S-\xi C\eps|\log\eps|+O(\eps), & N=4.
\end{cases}
\end{equation}
Hence, there exists $\eps_0>0$ such that
\begin{equation}\label{Aeps-less-S}
A_\eps<S\qquad\text{for all }0<\eps<\eps_0.
\end{equation}

Let
$g(t):=\frac{t^2}{2}A_\eps-\frac{t^{2^*}}{2^*}$ for $t\ge0$. Since
$A_\eps>0$ for $\eps>0$ small, the derivative
\[
g'(t)=t\bigl(A_\eps-t^{2^*-2}\bigr)
\]
vanishes at $t=0$ (a local minimum) and at
$t_\eps:=A_\eps^{1/(2^*-2)}$. Moreover,
$$
g''(t_\eps)=-(2^*-2)A_\eps<0,
$$
so $t_\eps$ is the global maximum
point (as $g(0)=0$ and $g(t)\to-\infty$ as $t\to+\infty$). Hence
\begin{equation}\label{eq:ray-polynomial-max}
\max_{t\ge0}g(t)=g(t_\eps)
=A_\eps^{\frac{2^*}{2^*-2}}\left(\frac12-\frac{1}{2^*}\right)
=\frac1N A_\eps^{N/2},
\end{equation}
where we used $\frac{2^*}{2^*-2}=\frac N2$ and
$\frac12-\frac{1}{2^*}=\frac1N$.

For all sufficiently small $\eps$, $A_\eps\le C_A$ for a
constant $C_A>0$ independent of $\eps$. By H\"older's inequality with conjugate exponents
$2^*/q_1$ and $2^*/(2^*-q_1)$, and $\|v_\eps\|_{2^*}=1$,
\[
\int_\Omega|v_\eps|^{q_1}\dd x
\le |\Omega|^{1-q_1/2^*}.
\]
Thus (F4) gives, for every $t\ge0$,
\[
\J(tU_\eps)\le\frac{C_A}{2}t^2
-\frac{t^{2^*}}{2^*}
+C_1|\Omega|^{1-q_1/2^*}t^{q_1}.
\]
Since $q_1<2<2^*$, the right-hand side tends to $-\infty$ as
$t\to\infty$ and is independent of $\eps$. Hence there exists
$C_0>0$, independent of $\eps$ and large enough so that
$S^{1/(2^*-2)}<C_0$, with the property that $\J(tU_\eps)<0$ for all
$t>C_0$. Consequently, it follows from $\J(0)=0$ that
\begin{equation}\label{eq:supJ1}
\sup_{t\ge0}\J(tU_\eps)
=\sup_{0\le t\le C_0}\J(tU_\eps).
\end{equation}

We now distinguish $N=4$ and $N>4$.

\smallskip

\textbf{Part 1: $N=4$.}
By Lemma \ref{lem:negative-N4}, uniformly for $0\le t\le C_0$,
\begin{equation}\label{eq:JlowerA}
\J(tU_\eps)
\le\frac{t^2}{2}A_\eps-\frac{t^4}{4}
+o(\eps|\log\eps|).
\end{equation}
By Lemma \ref{lem:classical},
\[
A_\eps=S-\xi C\eps|\log\eps|+O(\eps),\qquad C>0.
\]
By coercivity and \eqref{Aeps-less-S}, $0<A_\eps<S$
for small $\eps$. For $N=4$, \eqref{eq:ray-polynomial-max}
gives $\max_{t\ge0}g(t)=A_\eps^2/4$, attained at
$t_\eps=A_\eps^{1/2}$.
Combining \eqref{eq:supJ1} and \eqref{eq:JlowerA}, we get
\[
\sup_{t\ge0}\J(tU_\eps)
\le\frac14 A_\eps^2+o(\eps|\log\eps|).
\]
Moreover,
\begin{align*}
A_\eps^2
&=S^2+2S(A_\eps-S)+(A_\eps-S)^2\\
&=S^2-2S\xi C\eps|\log\eps|
+o(\eps|\log\eps|),
\end{align*}
because $\eps=o(\eps|\log\eps|)$ and
$(\eps|\log\eps|)^2=o(\eps|\log\eps|)$.
Therefore
\[
\sup_{t\ge0}\J(tU_\eps)
\le\frac14 S^2-\frac{S\xi C}{2}\eps|\log\eps|
+o(\eps|\log\eps|)
<\frac14 S^2
\]
for all sufficiently small $\eps>0$. This proves
\eqref{energy-cond} in dimension four.

\smallskip

\textbf{Part 2: $N>4$.}
With the same definition of $A_\eps$, we have
\[
\J(tU_\eps)=\frac{t^2}{2}A_\eps-\frac{t^{2^*}}{2^*}
-\int_\Omega F(x,tv_\eps\mathbf e)\dd x,\qquad t\ge0.
\]
Let $q\in(\max\{q_1,4/(N-2)\},2)$ and $C_F>0$ be as in
the proof of Lemma \ref{lem:negative-potential}. By
\eqref{eq:negative-power}, for $0\le t\le C_0$,
\[
-\int_\Omega F(x,tv_\eps\mathbf e)\dd x
\le\int_\Omega F^-(x,tv_\eps\mathbf e)\dd x
\le C_Ft^q\int_\Omega|v_\eps|^q\dd x
\le C_FC_0^q\int_\Omega|v_\eps|^q\dd x.
\]
By Lemma \ref{lem:Lq},
$\int_\Omega|v_\eps|^q\dd x=o(\eps)$, so
\[
-\int_\Omega F(x,tv_\eps\mathbf e)\dd x\le o(\eps)
\qquad\text{uniformly in }t\in[0,C_0].
\]
Therefore, for all $t\in[0,C_0]$,
\begin{equation}\label{eq:supJ2}
\J(tU_\eps)\le\frac{t^2}{2}A_\eps-\frac{t^{2^*}}{2^*}+o(\eps)
\le\frac1N A_\eps^{N/2}+o(\eps),
\end{equation}
where we used \eqref{eq:ray-polynomial-max}. The maximum is attained at
$$
t_\eps=A_\eps^{1/(2^*-2)}\to S^{1/(2^*-2)}<C_0,
$$
for $\eps>0$ small.
Combining with \eqref{eq:supJ1} and \eqref{eq:supJ2} gives
\begin{equation}\label{eq:JtU-01}
\sup_{t\ge0}\J(tU_\eps)\le\frac1N A_\eps^{N/2}+o(\eps).
\end{equation}
By Lemma \ref{lem:classical} (or see \eqref{AS-01}),
$$
A_\eps=S-\xi C\eps+o(\eps), \quad C>0, \ \xi>0.
$$
Hence
\begin{equation}\label{eq:AS-01}
S-A_\eps\ge\frac{\xi C}{2}\eps
\end{equation}
for $\eps>0$ sufficiently
small. By the mean value theorem applied to the function
$s\mapsto s^{N/2}$ on the interval $[A_\eps,S]$,
\begin{equation}\label{eq:AS-02}
S^{N/2}-A_\eps^{N/2}
=\frac{N}{2}\theta_\eps^{\frac N2-1}(S-A_\eps)
\end{equation}
for some $\theta_\eps\in(A_\eps,S)$. Since $A_\eps\to S$, we have
$\theta_\eps\ge S/2$ for $\eps>0$ sufficiently small. It follows from \eqref{eq:AS-01} and \eqref{eq:AS-02} that
\[
S^{N/2}-A_\eps^{N/2}
\ge\frac{N}{2}\left(\frac S2\right)^{\frac N2-1}(S-A_\eps)
\ge\frac{N\xi C}{4}\left(\frac S2\right)^{\frac N2-1}\eps.
\]
Setting
\[
c_2:=\frac{N\xi C}{4}\left(\frac S2\right)^{\frac N2-1}>0,
\]
we obtain
$$
A_\eps^{N/2}\le S^{N/2}-c_2\eps,
$$
for $\eps>0$
sufficiently small. Thus, it follows from \eqref{eq:JtU-01} that for $\eps>0$
sufficiently small,
\[
\sup_{t\ge0}\J(tU_\eps)\le\frac1N S^{N/2}-\frac1N c_2\eps+o(\eps)
<\frac1N S^{N/2}.
\]
Thus the energy threshold condition \eqref{energy-cond} holds for $N>4$.

\smallskip

In both cases, the energy threshold condition holds whenever
$0<\xi<\lambda_1$. Hence, Theorem \ref{thm:N4} follows from
Theorem \ref{thm:abstract}.
\end{proof}

\section{The case \texorpdfstring{$N=3$}{N = 3}: proof of Theorem \ref{thm:N3}}\label{sec:N3}

For $N\ge4$, the Sobolev quotient along the test function $U_\eps$
lies below $S$ by an amount of order $\eps$ (if $N>4$) or
$\eps|\log\eps|$ (if $N=4$), provided that
$\xi>0$
(see Section \ref{sec:N4} and Remark \ref{rem:reduction-symmetry}).
Together with the estimates for $F$, this gives the energy threshold
condition \eqref{energy-cond}.

For $N=3$, the cutoff error and the linear contribution both have
order $\eps^{1/2}$ (see the proof of Theorem \ref{thm:N3} below).
Their coefficients must be compared. We verify \eqref{energy-cond}
under the stronger sufficient condition
$\xi>\lambda_{\mathrm{BN}}$, where
$\lambda_{\mathrm{BN}}$ is given above Theorem \ref{thm:N3}.

For a ball, we normalize the function $\phi u_\eps$.
For a general bounded smooth domain, we construct concentrating
test functions using the Green function of $-\Delta-\xi$.
In both cases, the negative potential is estimated uniformly
along bounded portions of the normalized test rays.

\begin{lemma}\label{lem:negative-N3}
Let $N=3$ and assume (F1)--(F2) and (F4).
Suppose $v_\eps:\Omega\to\R$ is a family of measurable
functions such that $v_\eps\to0$ a.e.\ and
\[
\frac{|v_\eps(x)|^2}{\eps^{1/2}}\le h(x)
\quad\text{a.e. in }\Omega
\]
for some nonnegative $h\in L^1(\Omega)$ independent of
$\eps$. Then, for every unit vector $\mathbf e\in\R^2$
and every $T>0$,
\begin{equation}\label{eq:negative-N3}
\sup_{0\le t\le T}
\int_\Omega F^-(x,tv_\eps\mathbf e)\dd x
=o(\eps^{1/2}).
\end{equation}
\end{lemma}

\begin{proof}
For $s>0$, put
\[
\omega(s):=\sup_{\substack{x\in\Omega\\0<|U|\le s}}
\frac{F^-(x,U)}{|U|^2},\qquad \omega(0):=0.
\]
Lemma \ref{lem:negative-quadratic} gives
$0\le\omega(s)\le M$ and $\omega(s)\to0$ as $s\to0^+$.
The function $\omega$ is nondecreasing and hence measurable.
For $0\le t\le T$,
\[
0\le F^-(x,tv_\eps\mathbf e)
\le t^2|v_\eps(x)|^2\omega(T|v_\eps(x)|)
\le T^2|v_\eps(x)|^2\omega(T|v_\eps(x)|).
\]
It follows that
\[
0\le
\frac{\displaystyle\sup_{0\le t\le T}
\int_\Omega F^-(x,tv_\eps\mathbf e)\dd x}{\eps^{1/2}}
\le T^2\int_\Omega h(x)\omega(T|v_\eps(x)|)\dd x.
\]
The integrand on the right tends to zero a.e.\ and is bounded
by $Mh\in L^1(\Omega)$. The dominated convergence theorem
therefore proves \eqref{eq:negative-N3}.
\end{proof}

\begin{lemma}\label{lem:green-tests}
Let $N=3$ and $\lambda_{\mathrm{BN}}<\xi<\lambda_1$.
There exist $x_0\in\Omega$, $d_\xi>0$, and functions
$v_\eps\in H_0^1(\Omega)$, with $\|v_\eps\|_6=1$, such that
\begin{equation}\label{eq:green-quotient}
\int_\Omega\bigl(|\nabla v_\eps|^2-\xi v_\eps^2\bigr)\dd x
=S-d_\xi\eps^{1/2}+o(\eps^{1/2}).
\end{equation}
Moreover, $v_\eps\to0$ a.e.\ in $\Omega$, and there is a
constant $C>0$, independent of small $\eps$, such that
\begin{equation}\label{eq:green-pointwise}
\frac{|v_\eps(x)|^2}{\eps^{1/2}}
\le C\bigl(1+|x-x_0|^{-2}\bigr)
\qquad\text{for a.e. }x\in\Omega.
\end{equation}
\end{lemma}

\begin{proof}
Write
\[
S_\xi:=\inf_{w\in H_0^1(\Omega)\setminus\{0\}}
\frac{\displaystyle\int_\Omega(|\nabla w|^2-\xi w^2)\dd x}
{\|w\|_6^2}.
\]
By \cite[Corollary 1.1 and (1.51)]{BrezisNirenberg},
$S_\xi<S$ for $\xi>\lambda_{\mathrm{BN}}$.
The operator $-\Delta-\xi$ is coercive because $\xi<\lambda_1$.
Let $G_\xi$ denote its Dirichlet Green function, normalized by
\[
(-\Delta_x-\xi)G_\xi(x_0,x)=4\pi\delta_{x_0}
\quad\text{in }\Omega,\qquad
G_\xi(x_0,x)=0\quad\text{on }\partial\Omega,
\]
where $\delta_{x_0}$ is the Dirac mass at $x_0$. Set
\[
H_\xi(x_0,x):=\frac1{|x-x_0|}-G_\xi(x_0,x),
\qquad R_\xi(x_0):=H_\xi(x_0,x_0).
\]
The regular part extends continuously across $x=x_0$.
With this sign convention, \cite[Theorem 0.1]{Druet2002}
implies that there exists $x_0\in\Omega$ with
$R_\xi(x_0)<0$. Fix such a point.

Put
\[
b_\eps(x):=\frac{\eps^{1/4}}{(\eps+|x-x_0|^2)^{1/2}}.
\]
Let $Pb_\eps\in H_0^1(\Omega)$ solve
$-\Delta(Pb_\eps)=3b_\eps^5$, and let $H_0$ be the regular
part of the Dirichlet Green function of $-\Delta$, with the
same normalization. Define
\begin{equation}\label{eq:green-corrected-test}
z_\eps:=Pb_\eps-\eps^{1/4}
\bigl(H_\xi(x_0,\cdot)-H_0(x_0,\cdot)\bigr).
\end{equation}
The difference of the two regular parts belongs to
$H_0^1(\Omega)$, so $z_\eps\in H_0^1(\Omega)$.
This is the construction in
\cite[(2.1)]{FrankKonigKovarik2021}, with their potential
$a=-\xi$ and concentration parameter $\lambda=\eps^{-1/2}$.
The first-order energy expansions used below are contained in
\cite[Theorem 2.1]{FrankKonigKovarik2021}, with their perturbation
parameter set to zero. As stated after that theorem, these
expansions require coercivity, not criticality of $a$.
We give the first-order calculation and the pointwise estimate.

Since $-\Delta b_\eps=3b_\eps^5$, the function
\[
r_\eps:=b_\eps-Pb_\eps-\eps^{1/4}H_0(x_0,\cdot)
\]
is harmonic. On the boundary it equals
\[
\eps^{1/4}\left(
\frac1{(\eps+|x-x_0|^2)^{1/2}}-\frac1{|x-x_0|}\right).
\]
The distance from $x_0$ to $\partial\Omega$ is positive.
The mean value theorem and the maximum principle therefore give
$\|r_\eps\|_\infty=O(\eps^{5/4})$. Consequently,
\begin{equation}\label{eq:green-pointwise-expansion}
z_\eps=b_\eps-\eps^{1/4}H_\xi(x_0,\cdot)-r_\eps.
\end{equation}
For this fixed $x_0$, $H_\xi(x_0,\cdot)$ is bounded. Hence
\[
|z_\eps(x)|\le
C\eps^{1/4}\bigl(1+|x-x_0|^{-1}\bigr),
\qquad z_\eps(x)\to0\quad(x\ne x_0).
\]

Let
\[
B:=\int_{\R^3}(1+|y|^2)^{-3}\dd y=\frac{\pi^2}{4}.
\]
The scalar Sobolev constant is $S=3B^{2/3}$, as follows by
testing with $(1+|y|^2)^{-1/2}$ and using its equation.
The change of variables $x=x_0+\eps^{1/2}y$ gives
\[
\int_\Omega b_\eps^6\dd x=B+O(\eps^{3/2}),\qquad
\int_\Omega b_\eps^4\dd x=O(\eps^{1/2}).
\]
By continuity and boundedness of $H_\xi(x_0,\cdot)$,
dominated convergence also gives
\begin{align*}
\eps^{-1/4}\int_\Omega H_\xi(x_0,x)b_\eps^5\dd x
&\longrightarrow R_\xi(x_0)
\int_{\R^3}(1+|y|^2)^{-5/2}\dd y\\
&=\frac{4\pi}{3}R_\xi(x_0).
\end{align*}
Here the last integral equals $4\pi/3$, by polar coordinates.
Expanding the sixth power in \eqref{eq:green-pointwise-expansion},
the terms of degree at least two in $z_\eps-b_\eps$ are
bounded in absolute value by
\[
C\eps^{1/2}\int_\Omega b_\eps^4\dd x+C\eps^{3/2}
=O(\eps).
\]
The term containing $r_\eps b_\eps^5$ is $O(\eps^{3/2})$.
It follows that
\begin{equation}\label{eq:green-denominator}
\int_\Omega|z_\eps|^6\dd x
=B-8\pi R_\xi(x_0)\eps^{1/2}+o(\eps^{1/2}).
\end{equation}

To compute the numerator, note that
$-\Delta H_\xi(x_0,\cdot)=-\xi G_\xi(x_0,\cdot)$ and
$-\Delta H_0(x_0,\cdot)=0$ in the distributional sense.
Thus
\[
-\Delta z_\eps=3b_\eps^5+\xi\eps^{1/4}G_\xi(x_0,\cdot),
\]
and testing with $z_\eps$ yields
\begin{align*}
\int_\Omega(|\nabla z_\eps|^2-\xi z_\eps^2)\dd x
&=3\int_\Omega b_\eps^5z_\eps\dd x\\
&\quad+\xi\int_\Omega
\bigl(\eps^{1/4}G_\xi(x_0,x)-z_\eps(x)\bigr)z_\eps(x)\dd x.
\end{align*}
Writing $r=|x-x_0|$, set
\[
D_\eps(x):=\frac{\eps^{1/4}}r-b_\eps(x).
\]
For $r>0$,
\[
0\le D_\eps(x)
\le\min\left\{\frac{\eps^{1/4}}r,
\frac{\eps^{5/4}}{2r^3}\right\},
\qquad
b_\eps(x)\le\min\{\eps^{-1/4},\eps^{1/4}/r\}.
\]
Choose $R_0>0$ with $\Omega\subset B_{R_0}(x_0)$.
Splitting the radial integrals at $r=\eps^{1/2}$, we obtain
\begin{align*}
\int_\Omega D_\eps b_\eps\dd x
&\le C\int_0^{\eps^{1/2}}r\dd r
+C\eps^{3/2}\int_{\eps^{1/2}}^{R_0}r^{-2}\dd r
=O(\eps),\\
\int_\Omega D_\eps\dd x
&\le C\eps^{1/4}\int_0^{\eps^{1/2}}r\dd r
+C\eps^{5/4}\int_{\eps^{1/2}}^{R_0}r^{-1}\dd r
=O(\eps^{5/4}|\log\eps|).
\end{align*}
Since $\eps^{1/4}G_\xi-z_\eps=D_\eps+r_\eps$ and
$|z_\eps|\le b_\eps+C\eps^{1/4}$, these bounds imply
\[
\left|\int_\Omega
(\eps^{1/4}G_\xi-z_\eps)z_\eps\dd x\right|=O(\eps).
\]
Using \eqref{eq:green-pointwise-expansion} in the remaining
integral, we conclude that
\begin{equation}\label{eq:green-numerator}
\int_\Omega(|\nabla z_\eps|^2-\xi z_\eps^2)\dd x
=3B-4\pi R_\xi(x_0)\eps^{1/2}+o(\eps^{1/2}).
\end{equation}

Set $v_\eps=z_\eps/\|z_\eps\|_6$.
By \eqref{eq:green-denominator}, $\|z_\eps\|_6^2\to B^{1/3}>0$.
Dividing \eqref{eq:green-numerator} by
$\|z_\eps\|_6^2$, we find
\begin{align*}
\int_\Omega(|\nabla v_\eps|^2-\xi v_\eps^2)\dd x
&=\frac{3B-4\pi R_\xi(x_0)\eps^{1/2}+o(\eps^{1/2})}
{B^{1/3}\left(1-\frac{8\pi R_\xi(x_0)}{3B}\eps^{1/2}
+o(\eps^{1/2})\right)}\\
&=3B^{2/3}+4\pi B^{-1/3}R_\xi(x_0)\eps^{1/2}
+o(\eps^{1/2}).
\end{align*}
Since $R_\xi(x_0)<0$, this proves \eqref{eq:green-quotient}
with $d_\xi=-4\pi B^{-1/3}R_\xi(x_0)>0$.
The pointwise bound for $z_\eps$ and the lower bound on
$\|z_\eps\|_6$ give \eqref{eq:green-pointwise} and
$v_\eps\to0$ a.e.\ Finally,
$1+|x-x_0|^{-2}\in L^1(\Omega)$ in dimension three.
\end{proof}

\begin{proof}[{\bf Proof of Theorem \ref{thm:N3}}] Let $\mathbf e$ be a unit eigenvector of $A^*$ corresponding
to $\xi$. We distinguish two cases for the domain $\Omega$.
 
\smallskip
\textbf{Case 1: $\Omega$ is a ball.}
After a translation, write $\Omega=B_R(0)$. Let $u_\eps$ be the Aubin--Talenti function  given in \eqref{u11} with $N=3$, and let $\phi$ be a fixed smooth radial
function such that $\phi(0)=1$, $\phi'(0)=0$, and $\phi(R)=0$. Set
\begin{equation}\label{def-w-eps-N3}
w_\eps:=\phi u_\eps=\frac{\phi(|x|)}{(\eps+|x|^2)^{1/2}},\quad \eps>0.
\end{equation}

For $N=3$ (so $2^*=6$), the classical estimates
\cite[Lemma 1.3]{BrezisNirenberg} give, as $\eps\to0^+$,
\begin{align}
\int_\Omega|\nabla w_\eps|^2\dd x
&=\frac{\kappa_1}{\eps^{1/2}}
+\omega_3\int_0^R|\phi'(r)|^2dr+O(\eps^{1/2}),
\label{BN-constants}\\
\|w_\eps\|_{L^{6}}^2
&=\frac{\kappa_2}{\eps^{1/2}}+O(\eps^{1/2}),
\label{BN-constants-1}\\
\int_\Omega|w_\eps|^2\dd x
&=\omega_3\int_0^R\phi^2(r)dr+O(\eps^{1/2}),
\label{BN-constants-2}
\end{align}
with $\kappa_1/\kappa_2=S$ and $\omega_3$ the area of the unit sphere
$S^2$.

For the test function $W_\eps:=w_\eps\mathbf e$, by
\eqref{BN-constants}--\eqref{BN-constants-2},
\begin{equation}\label{eq:quadratic-expansion}
\begin{aligned}
&\int_\Omega|\nabla W_\eps|^2\dd x-\int_\Omega A^*W_\eps\cdot W_\eps\dd x\\
&=\int_\Omega|\nabla w_\eps|^2\dd x-\xi\int_\Omega|w_\eps|^2\dd x\\
&=\frac{\kappa_1}{\eps^{1/2}}
+\omega_3\int_0^R|\phi'|^2\dd r
-\xi\omega_3\int_0^R\phi^2\dd r
+O(\eps^{1/2}),
\end{aligned}
\end{equation}
and
\[
\frac{1}{\|w_\eps\|_{L^{6}}^2}
=\frac{\eps^{1/2}}{\kappa_2}\bigl(1+O(\eps)\bigr).
\]
Multiplying,
\begin{equation}\label{BN-quotient}
\begin{aligned}
Q_\eps
&:=\frac{\displaystyle\int_\Omega|\nabla w_\eps|^2\dd x
-\xi\int_\Omega|w_\eps|^2\dd x}
{\|w_\eps\|_{L^{6}}^2}\\
&=\left(\frac{\kappa_1}{\eps^{1/2}}+\omega_3\int_0^R|\phi'|^2\dd r
-\xi\omega_3\int_0^R\phi^2\dd r+O(\eps^{1/2})\right)
\cdot\frac{\eps^{1/2}}{\kappa_2}\bigl(1+O(\eps)\bigr)\\
&=S+\eps^{1/2}\cdot\frac{\omega_3}{\kappa_2}
\left[\int_0^R|\phi'(r)|^2\dd r
-\xi\int_0^R\phi^2(r)\dd r\right]+O(\eps).
\end{aligned}
\end{equation}
Setting
\begin{equation}\label{c1-C-def}
c_1:=\frac{\omega_3}{\kappa_2}\int_0^R|\phi'(r)|^2dr,
\qquad
C:=\frac{\omega_3}{\kappa_2}\int_0^R\phi^2(r)dr,
\end{equation}
we may rewrite \eqref{BN-quotient} as
\begin{equation}\label{BN-quotient-c1C}
Q_\eps=S+\eps^{1/2}(c_1-\xi C)+O(\eps).
\end{equation}
Choosing $\phi(r)=\cos(\pi r/(2R))$, we have
\[
\int_0^R\phi^2(r)\dd r=\frac R2,\qquad
\int_0^R|\phi'(r)|^2\dd r=\frac{\pi^2}{8R}.
\]
Since $\lambda_1(B_R)=\pi^2/R^2$, it follows that
\begin{equation}\label{lambda4}
\frac{c_1}{C}=\frac{\pi^2/(8R)}{R/2}=\frac{\pi^2}{4R^2}=\frac{\lambda_1}{4}=\lambda_{\mathrm{BN}}.
\end{equation}
Hence, if $\xi>\lambda_{\mathrm{BN}}=\frac{\lambda_1}{4}$, we have
$c_1-\xi C<0$, and \eqref{BN-quotient-c1C} implies
\begin{equation}\label{Qeps-less-S}
Q_\eps<S
\end{equation}
for $\eps>0$ sufficiently small.

By coercivity, $Q_\eps>0$. Put
\[
v_\eps:=\frac{w_\eps}{\|w_\eps\|_6},\qquad
V_\eps:=v_\eps\mathbf e.
\]
Then $\|v_\eps\|_6=1$ and
\[
\J(tV_\eps)=\frac{t^2}{2}Q_\eps-\frac{t^6}{6}
-\int_\Omega F(x,tv_\eps\mathbf e)\dd x.
\]
By \eqref{BN-constants-1}, $\|w_\eps\|_6\ge c\eps^{-1/4}$
for small $\eps$. Hence, for $x\ne0$,
\[
|v_\eps(x)|\le
\frac{C\eps^{1/4}|\phi(|x|)|}{(\eps+|x|^2)^{1/2}},
\qquad
\frac{|v_\eps(x)|^2}{\eps^{1/2}}
\le\frac{C\phi(|x|)^2}{|x|^2}.
\]
In particular, $v_\eps\to0$ a.e.\ and the right-hand side
is integrable on $B_R(0)$, since
\[
\int_{B_R(0)}\frac{\phi(|x|)^2}{|x|^2}\dd x
=\omega_3\int_0^R\phi(r)^2\dd r<\infty.
\]
Lemma \ref{lem:negative-N3} therefore applies.

We also need a bound on the ray parameter independent of
$\eps$. Since $Q_\eps\to S$, choose $C_Q>0$ such that
$Q_\eps\le C_Q$ for small $\eps$. By H\"older's inequality,
\[
\int_\Omega|v_\eps|^{q_1}\dd x
\le |\Omega|^{1-q_1/6}
\left(\int_\Omega|v_\eps|^6\dd x\right)^{q_1/6}
=|\Omega|^{1-q_1/6}.
\]
Thus (F4) gives
\begin{equation}\label{eq:N3-ray-bound}
\J(tV_\eps)\le\frac{C_Q}{2}t^2-\frac{t^6}{6}
+C_1|\Omega|^{1-q_1/6}t^{q_1}.
\end{equation}
Since $q_1<2<6$, the right-hand side is negative for all
$t>T_0$, for some $T_0>0$ independent of small $\eps$.
As $\J(0)=0$,
\[
\sup_{t\ge0}\J(tV_\eps)=\sup_{0\le t\le T_0}\J(tV_\eps).
\]
By Lemma \ref{lem:negative-N3},
\[
\J(tV_\eps)\le\frac{t^2}{2}Q_\eps-\frac{t^6}{6}
+o(\eps^{1/2})
\qquad\text{uniformly for }0\le t\le T_0.
\]
The derivative of $g(t)=t^2Q_\eps/2-t^6/6$ is
$g'(t)=t(Q_\eps-t^4)$.
Its positive maximum point is $t_\eps=Q_\eps^{1/4}$, and
\[
g(t_\eps)=Q_\eps^{3/2}\left(\frac12-\frac16\right)
=\frac13Q_\eps^{3/2}.
\]
Consequently,
\begin{equation}\label{J-less-S1}
\sup_{t\ge0}\J(tV_\eps)
\le\frac13Q_\eps^{3/2}+o(\eps^{1/2}).
\end{equation}
By \eqref{BN-quotient-c1C} and \eqref{lambda4},
\[
Q_\eps=S-d\eps^{1/2}+O(\eps),
\qquad d:=C\left(\xi-\frac{\lambda_1}{4}\right)>0.
\]
The mean value theorem gives
$S^{3/2}-Q_\eps^{3/2}=\frac32\theta_\eps^{1/2}(S-Q_\eps)$
for some $\theta_\eps\in(Q_\eps,S)$, with $\theta_\eps\to S$.
Therefore \eqref{J-less-S1} yields
\[
\sup_{t\ge0}\J(tV_\eps)
\le\frac13S^{3/2}-\frac{S^{1/2}d}{2}\eps^{1/2}
+o(\eps^{1/2})<\frac13S^{3/2}
\]
for sufficiently small $\eps>0$.

\smallskip
\textbf{Case 2: $\Omega$ is a general bounded smooth domain.}
Let $v_\eps$ be the normalized functions in
Lemma \ref{lem:green-tests}, and put $V_\eps=v_\eps\mathbf e$.
Then
\[
Q_\eps:=\int_\Omega(|\nabla v_\eps|^2-\xi v_\eps^2)\dd x
=S-d_\xi\eps^{1/2}+o(\eps^{1/2}),\qquad d_\xi>0.
\]
Moreover, $v_\eps\to0$ a.e.\ and
\eqref{eq:green-pointwise} provides an integrable majorant
for $|v_\eps|^2/\eps^{1/2}$. By Lemma \ref{lem:negative-N3},
\[
\sup_{0\le t\le T}\int_\Omega F^-(x,tv_\eps\mathbf e)\dd x
=o(\eps^{1/2})
\qquad\text{for every fixed }T>0.
\]
Since $\|v_\eps\|_6=1$ and $Q_\eps\to S$, the calculation
in \eqref{eq:N3-ray-bound} applies to this family as well.
There is thus $T_0>0$, independent of small $\eps$, such that
\[
\sup_{t\ge0}\J(tV_\eps)=\sup_{0\le t\le T_0}\J(tV_\eps).
\]
Coercivity gives $Q_\eps>0$. Maximizing the polynomial part
as in Case 1, and using the uniform negative-part estimate, we
obtain
\begin{align*}
\sup_{t\ge0}\J(tV_\eps)
&\le\max_{t\ge0}\left(\frac{t^2}{2}Q_\eps-\frac{t^6}{6}\right)
+o(\eps^{1/2})\\
&=\frac13Q_\eps^{3/2}+o(\eps^{1/2})\\
&=\frac13S^{3/2}-\frac{S^{1/2}d_\xi}{2}\eps^{1/2}
+o(\eps^{1/2})
<\frac13S^{3/2}
\end{align*}
for all sufficiently small $\eps>0$.

\smallskip
The test functions $V_\eps=v_\eps\mathbf e$ apply to every
symmetric coupling matrix satisfying the stated spectral
conditions, irrespective of the signs of its entries.
In both cases, \eqref{energy-cond} holds whenever
$\lambda_{\mathrm{BN}}<\xi<\lambda_1$.
Theorem \ref{thm:abstract} now proves Theorem \ref{thm:N3}.
\end{proof}

\section{Nonexistence: proof of Proposition \ref{prop:nonexistence}}
\label{sec:nonexistence}

We first establish the regularity needed to apply the Pohozaev
identity to weak solutions. The following argument is the vector
form of the Brezis--Kato iteration \cite{BrezisKato}.

\begin{lemma}\label{lem:weak-regularity}
Assume (F1)--(F3). Every weak solution of \eqref{sys1}
belongs to $W^{2,s}(\Omega;\R^2)$ for every finite $s>1$, and hence
to $C^1(\overline\Omega;\R^2)$.
\end{lemma}
\begin{proof}
Write $r=|U|$ and
$G(x,U)=A^*U+|U|^{2^*-2}U+\nabla F(x,U)$.
By Lemma \ref{lem:growth2}, there is a constant $C>0$ such that
\[
 |G(x,U)|\le C\bigl(|U|+|U|^{2^*-1}\bigr).
\]
Thus $|G(x,U(x))|\le a(x)r(x)$, where
\[
 a(x)=C(1+r^{2^*-2})\in L^{N/2}(\Omega).
\]
For $L>0$ and $s\ge1$, let $r_L=\min\{r,L\}$ and
$w_L=r r_L^{s-1}$.
The functions $u_i r_L^{2s-2}$ belong to $H_0^1(\Omega)$,
since $z\mapsto z_i\min\{|z|,L\}^{2s-2}$ is Lipschitz
and vanishes at zero.
Test the two equations with $u_i r_L^{2s-2}$ and add them. Since
$|\nabla r|\le|\nabla U|$, the left-hand side is
\[
 \int_\Omega r_L^{2s-2}|\nabla U|^2\dd x
 +(2s-2)\int_\Omega r r_L^{2s-3}
             \nabla r\cdot\nabla r_L\dd x
 \ge \frac1{s^2}\int_\Omega|\nabla w_L|^2\dd x.
\]
For $s=1$, the second term is understood as zero. The inequality
follows by considering $r<L$ and $r\ge L$: on the first set,
$|\nabla w_L|^2=s^2r^{2s-2}|\nabla r|^2$; on the second,
$|\nabla w_L|^2=L^{2s-2}|\nabla r|^2$.
The right-hand side is at most $\int_\Omega a w_L^2\dd x$.
The Sobolev inequality gives
\[
 S\|w_L\|_{2^*}^2\le s^2\int_\Omega a w_L^2\dd x.
\]
Choose $M=M(s)$ so large that
$\|a\mathbf1_{\{a>M\}}\|_{N/2}<S/(2s^2)$. Splitting the
integral and applying H\"older's inequality yields
\[
 \frac S2\|w_L\|_{2^*}^2\le s^2M\|w_L\|_2^2.
\]
If $r\in L^{2s}$, the right-hand side is bounded independently
of $L$. Letting $L\to\infty$ gives $r\in L^{s2^*}$.
Starting with $r\in L^{2^*}$ and repeating this step proves
$r\in L^t$ for every finite $t\ge1$. Consequently
$G(x,U)\in L^s$ for every finite $s$. Global Dirichlet regularity
gives $U\in W^{2,s}$; choosing $s>N$ proves the last assertion.
\end{proof}

\begin{proof}[{\bf Proof of Proposition \ref{prop:nonexistence}}]
Let $U$ be a weak solution and put $X=x-x_0$. By Lemma \ref{lem:weak-regularity}, the
following integrations by parts are valid. They may also be
obtained by approximation in $W^{2,s}$ with $s>N$.
Define the total potential
\[
 H(x,U)=\frac12 A^*U\cdot U+\frac1{2^*}|U|^{2^*}+F(x,U).
\]
The system is $-\Delta U=\nabla_UH(x,U)$. Multiplication by
$X\cdot\nabla U$ and summation over the components give the
Pohozaev identity (see \cite{BrezisNirenberg,AlvesMoraisSouto2000})
\begin{equation}\label{eq:pohozaev-full}
\begin{aligned}
\frac{N-2}{2}\int_\Omega|\nabla U|^2\,dx
&+\frac12\int_{\partial\Omega}
(X\cdot\nu)|\partial_\nu U|^2\,d\sigma\\
&=N\int_\Omega H(x,U)\,dx
+\int_\Omega X\cdot\nabla_xH(x,U)\,dx.
\end{aligned}
\end{equation}
Here $\nu$ is the outward unit normal. Indeed, for each component,
\[
 \int_\Omega(-\Delta u_i)(X\cdot\nabla u_i)\dd x
 =-\frac{N-2}{2}\int_\Omega|\nabla u_i|^2\dd x
  -\frac12\int_{\partial\Omega}(X\cdot\nu)
                                  |\partial_\nu u_i|^2\dd\sigma.
\]
On the other hand, the chain rule gives
\[
 \nabla_UH(x,U)\cdot(X\cdot\nabla U)
 =X\cdot\nabla[H(x,U(x))]-X\cdot\nabla_xH(x,U).
\]
Since $H(x,0)=0$ and $U=0$ on the boundary, integration of the
first term on the right gives $-N\int_\Omega H(x,U)\dd x$.
This proves \eqref{eq:pohozaev-full}.

Testing \eqref{sys1} with $U$ gives
\begin{equation}\label{eq:pohozaev-test}
 \int_\Omega|\nabla U|^2\dd x
 =\int_\Omega A^*U\cdot U\dd x
  +\int_\Omega|U|^{2^*}\dd x
  +\int_\Omega\nabla F(x,U)\cdot U\dd x.
\end{equation}
Substitute \eqref{eq:pohozaev-test} into \eqref{eq:pohozaev-full}.
The critical terms cancel because $N/2^*=(N-2)/2$, whereas
$N/2-(N-2)/2=1$. Thus
\begin{equation}\label{eq:pohozaev-reduced}
 \frac12\int_{\partial\Omega}(X\cdot\nu)
                         |\partial_\nu U|^2\dd\sigma
 =\int_\Omega\bigl[A^*U\cdot U+\mathcal P_F(x,U)\bigr]\dd x.
\end{equation}
Star-shapedness gives $X\cdot\nu\ge0$ on the boundary.
Since $A^*U\cdot U\le\xi|U|^2$, condition \eqref{eq:pohozaev-bound} implies
\[
 0\le\frac12\int_{\partial\Omega}(X\cdot\nu)
                              |\partial_\nu U|^2\dd\sigma
 \le(\xi+\kappa_F)\int_\Omega|U|^2\dd x.
\]
If $\xi+\kappa_F<0$, this forces $U=0$.
\end{proof}

\medskip
\noindent\textbf{Ethics Declaration.}
Not applicable.

\medskip
\noindent\textbf{Conflict of Interest.}
The author declares that there is no conflict of interest.

\medskip
\noindent\textbf{Data Availability.}
Data sharing is not applicable to this article, as no
datasets were generated or analyzed during the current study.

\end{document}